\documentclass[11pt]{amsart}
\usepackage{amsfonts,amssymb,amsmath,amsthm}
\usepackage[colorlinks,citecolor=blue,urlcolor=black,linkcolor=black]{hyperref}

\newtheorem{thm}{Theorem}[section]
\newtheorem{mainthm}{Theorem}

\newtheorem{lemma}[thm]{Lemma}
\newtheorem{cor}[thm]{Corollary}
\newtheorem{claim}{Claim}[thm]
\newtheorem{prop}[thm]{Proposition}
\newtheorem{obs}[thm]{Observation}
\newtheorem{fact}[thm]{Fact}

\theoremstyle{definition}
\newtheorem{defn}[thm]{Definition}

\theoremstyle{remark}
\newtheorem{remark}[thm]{Remark}
\newtheorem{conv}[thm]{Convention}

\renewcommand{\mid}{\mathrel{|}\allowbreak}
\renewcommand{\restriction}{\mathbin\upharpoonright}
\DeclareMathOperator{\add}{Add}
\DeclareMathOperator{\acc}{acc}
\DeclareMathOperator{\nacc}{nacc}
\DeclareMathOperator{\ch}{CH}
\DeclareMathOperator{\reg}{Reg}
\DeclareMathOperator{\cf}{cf}
\DeclareMathOperator{\dom}{dom}
\DeclareMathOperator{\im}{Im}
\DeclareMathOperator{\otp}{otp}
\DeclareMathOperator{\h}{ht}
\DeclareMathOperator{\p}{P}

\newcommand\s{\subseteq}
\newcommand\sq{\sqsubseteq}

\newcommand*\axiomfont[1]{\textsf{\textup{#1}}}
\newcommand\gch{\axiomfont{GCH}}
\newcommand\sh{\axiomfont{SH}}
\newcommand\zfc{\axiomfont{ZFC}}
\def\sd{\framebox[3.0mm][l]{$\diamondsuit$}\hspace{0.5mm}{}}
\def\br{\blacktriangleright}

\title{Full Souslin trees at small cardinals}

\author{Assaf Rinot}
\address{Department of Mathematics, Bar-Ilan University, Ramat-Gan 52900, Israel.}
\urladdr{http://www.assafrinot.com}

\author{Shira Yadai}
\address{Department of Mathematics, Bar-Ilan University, Ramat-Gan 52900, Israel.}
\email{greenss@biu.ac.il}

\author{Zhixing You}
\address{Department of Mathematics, Bar-Ilan University, Ramat-Gan 52900, Israel.}
\email{zhixingy121@gmail.com}


\begin{document}
\begin{abstract} A $\kappa$-tree is \emph{full} if each of its limit levels omits no more than one potential branch.
Kunen asked whether a full $\kappa$-Souslin tree may consistently exist.
Shelah gave an affirmative answer of height a strong limit Mahlo cardinal.
Here, it is shown that these trees may consistently exist at small cardinals.
Indeed, there can be $\aleph_3$ many full $\aleph_2$-trees such that the product of any countably many of them is an $\aleph_2$-Souslin tree.

\end{abstract}
\date{A preprint as of February 29, 2023. For updates, visit \textsf{http://p.assafrinot.com/62}.}
\maketitle

\section{Introduction}
Recall that the real line is the unique separable, dense linear ordering with no endpoints in which
every bounded set has a least upper bound.
A problem posed by Mikhail Souslin more than a century ago and published in the very first volume of \emph{Fundamenta Mathematicae} \cite{Souslin} asks whether
the term \emph{separable} in the above characterization may be weakened to \emph{ccc}.\footnote{A linear order is \emph{separable}
if it has a countable dense subset. It is \emph{ccc} (short for \emph{countable chain condition}) if every pairwise disjoint family of open intervals is countable.}
The affirmative answer to Souslin's problem --- equivalently,
that every linearly ordered topological space satisfying the countable chain condition is separable ---
is known as Souslin's Hypothesis, and abbreviated \sh.

In the early 1930's, in the course of attempting to settle Souslin's problem,
Kurepa discovered that the problem can be reformulated in terms of transfinite trees
and thus
``eliminated topological considerations from Souslin's Problem and
reduced it to a problem of combinatorial set theory''~\cite[p.~3]{MR2882649}.
Kurepa's finding asserts that $\sh$ is equivalent to a
Ramsey-theoretic statement concerning transfinite trees: Every uncountable tree
must admit an uncountable chain or an uncountable antichain.
Curiously, around the same time, Sierpi{\'n}ski \cite{MR1556708} proved that there does exist an uncountable
\emph{partial order} without uncountable chains or antichains. So, it is the requirement of being a tree that makes this problem difficult.

What are transfinite trees anyway?
Intuitively, the class of these trees should be understood as a two-step generalization of the class of finite linear orders, where the first step of generalization is the class of well-ordered sets.
Precisely, for an infinite cardinal $\kappa$, a partially ordered set $\mathbf T=(T,{<_T})$ is a \emph{$\kappa$-tree} iff the following two requirements hold:
\begin{enumerate}
\item For every node $x\in T$, the set $x_\downarrow:=\{ y\in T\mid y<_T x\}$ is well-ordered by $<_T$.
Hereafter, write $\h(x):=\otp(x_\downarrow,<_T)$ for the \emph{height} of  $x$;
\item For every $\alpha<\kappa$, the $\alpha^{\text{th}}$-level of the tree, $T_\alpha:=\{ x\in T\mid \h(x)=\alpha\}$, is nonempty and has size less than $\kappa$.
The level $T_\kappa$ is empty.
\end{enumerate}

For an ordinal $\alpha$, a subset $B\s T$ is an \emph{$\alpha$-branch} iff $(B,<_T)$ is linearly ordered and $\{\h(x)\mid x\in B\}=\alpha$.
With this language, K{\H o}nig's infinity lemma \cite{konig1927schlussweise} can be restated as asserting that every $\aleph_0$-tree admits an $\aleph_0$-branch,
and Kurepa's theorem \cite{kurepa1935ensembles} is that  $\sh$ can be restated as asserting that every $\aleph_1$-tree with no uncountable antichains
must admit an $\aleph_1$-branch. A counterexample tree is called an \emph{$\aleph_1$-Souslin tree}.

Souslin's problem was eventually resolved at the end of the 1960's \cite{jech1967non,tennenbaum1968souslin,jensen1968souslin,MR0294139}
with the finding that $\sh$ is independent of the usual axioms of set theory ($\zfc$).
Amazingly enough, the resolution of this single problem led to profound discoveries in set theory:
the notions of Aronszajn, Kurepa, and Souslin trees \cite{kurepa1935ensembles},
forcing axioms and the method of iterated forcing \cite{MR0294139},
the combinatorial principles of G\"odel's universe of sets \cite{jen72},
and the theory of iteration without adding reals \cite{MR384542}.

Even before $\aleph_1$-Souslin trees were known to consistently exist,
Rudin \cite{Rudin_souslin_line_dowker_space} famously used them to give a (conditional) construction of
a normal topological space whose product with the unit interval is not normal (a \emph{Dowker space} \cite{Dowker_C.H.}).
Ever since, the study of $\kappa$-Souslin trees (i.e., $\kappa$-trees with no $\kappa$-sized chains or antichains) and their applications
remained an active and fruitful vein of research in combinatorial set theory, general topology and functional analysis.
A very recent application of (homogeneous) Souslin trees to infinite group theory may be found in \cite{paper60}.
Meanwhile, fundamental questions on the existence of $\aleph_2$-Souslin trees remain open to this date (see the table on  \cite[p.~439]{paper31} for a summary of the current state).

\smallskip

A great deal of information about a $\kappa$-tree $\mathbf T$ is encoded
into the collection $\mathcal V(\mathbf T)$ of its vanishing branches,
where an $\alpha$-branch is said to be \emph{vanishing} iff it has no upper bound in $\mathbf T$.
In fact, one part of the proof of Kurepa's theorem mentioned earlier goes through showing that
if $\mathbf T$ is an $\aleph_1$-Souslin tree, then
$\mathcal V(\mathbf T)$
admits a natural lexicographic-like ordering that makes it into a ccc non-separable linear order.

In the early 1990's, Kunen asked whether a $\kappa$-Souslin tree may be full,
where a $\kappa$-tree is \emph{full} iff
it admits no more than one vanishing $\alpha$-branch for every limit ordinal $\alpha<\kappa$.
While it is hard to see how a full $\kappa$-tree could manage to evade having a $\kappa$-sized chain or an antichain,
Kunen's question was answered in the affirmative. Specifically, in \cite{Sh:624}, Shelah constructed a full $\kappa$-Souslin tree
for some `high up' cardinal $\kappa$ (a strong limit Mahlo cardinal, to be precise),  using the method of forcing.

In this paper, we shall give forcing-free constructions of full $\kappa$-Souslin trees.
Our combinatorial approach is based on the proxy principle $\boxtimes^-(\kappa)$ from \cite{paper22}
and a new prediction principle for trees (see $\S3$ below) which provably holds at subtle cardinals. In particular, we obtain the following.

\begin{mainthm}\label{thma} Suppose that $\kappa$ is a subtle cardinal and that $\boxtimes^-(\kappa)$ holds.
Then there exists a full $\kappa$-Souslin tree.
\end{mainthm}

The definition of full trees is quite illusive, and it is tempting to think that full Souslin trees can only exist at the level of strong limit cardinals.
Personally, we tried to prove that this must be the case, but only got as far as Observation~\ref{obs22}(2) below.
The second result of this paper proves that our initial intuition was wrong and shows that a full $\kappa$-Souslin tree may exist for $\kappa$ as low as $\aleph_n$ for some positive integer $n$.
\begin{mainthm}\label{thmb}
Suppose that $\lambda$ is the successor of an uncountable cardinal
and that $\square_\lambda$ and $\gch$ both hold. Then there exists a full $\lambda^+$-Souslin tree.
\end{mainthm}

In the other direction, we shall also show that
the existence of a full $\lambda^+$-Souslin tree is compatible with $\lambda$ being a supercompact cardinal.

\medskip

Finally, the definition of full $\kappa$-trees may suggest that if they exist, then they are unique (say, any two are isomorphic on a club).
This sounds even more plausible in the context of splitting binary $\kappa$-trees, i.e., trees $\mathbf T=(T,{\s})$ where $T$ is a downward-closed subset of ${}^{<\kappa}2$ such that $t{}^\smallfrown\langle0\rangle,t{}^\smallfrown\langle1\rangle\in T$ for every $t\in T$.
Nonetheless, our third main finding shows that this is not the case either.
By Proposition~\ref{weakpr} below, if $\mathbf S,\mathbf T$ are two $\kappa$-trees whose product is $\kappa$-Souslin,
then there is no weak embedding from $\mathbf S$ to $\mathbf T$. Therefore, the splitting, binary full trees given by the following theorem are pairwise distinct in a very strong sense.

\begin{mainthm}\label{thmc} Suppose that $\lambda$ is a regular uncountable cardinal
such that $\diamondsuit(\lambda)$ and $\sd_\lambda$ both hold.
Then there exists a family $\langle\mathbf T^i\mid i<2^{\lambda^{+}}\rangle$ of splitting, binary, full $\lambda^+$-Souslin trees
such that for every nonempty $I\in[2^{\lambda^{+}}]^{<\lambda}$, the product $\bigotimes_{i\in I}\mathbf T^i$ is again $\lambda^+$-Souslin.
\end{mainthm}

The preceding constitutes the main result of this paper and it is optimal in three ways. First, an $\aleph_1$-Souslin tree can never be full, and the theorem implies that full $\aleph_2$-Souslin trees may consistently exist.
Second, the product of $\lambda$-many $\lambda^+$-trees can never be a $\lambda^+$-tree, and the theorem successfully handles the product of less than $\lambda$ many  $\lambda^+$-Souslin trees.
Third, the number of pairwise non-weakly-embeddable full $\lambda^+$-Souslin trees constructed is $2^{\lambda^+}$ which is the largest possible.

\subsection{Organization of this paper} In Section~\ref{secfull}, we provide a few preliminaries on trees, $C$-sequences,
and the proxy principle.

In Section~\ref{secdia}, we study two new diamond-type prediction principles for trees, showing that the stronger one provably holds at subtle cardinals,
provably fails at successors of singulars, and consistently holds at all successors of regular uncountable.

In Section~\ref{inacc}, we deal with full $\kappa$-Souslin trees, for $\kappa$ a strongly inaccessible cardinal.
We start by giving the simplest combinatorial construction of a full $\kappa$-Souslin tree,
from which we obtain Theorem~\ref{thma}.
We then move on to constructing $2^\kappa$ many full $\kappa$-Souslin trees that are pairwise Souslin.

In Section~\ref{succ},
we deal with full $\kappa$-Souslin trees for $\kappa$ a successor of a regular.
To avoid repetitions, we start outright with the most general construction of a large family of full $\kappa$-Souslin trees,
from which we obtain Theorem~\ref{thmc}. We then explain how to obtain Theorem~\ref{thmb}.

\section{Preliminaries}\label{secfull}
Throughout this paper, $\kappa$ denotes a regular uncountable cardinal,
and $\lambda$ denotes an infinite cardinal.
$H_\kappa$ denotes  the collection of all sets of hereditary cardinality less than $\kappa$.
$\reg(\kappa)$ denotes the set of all infinite regular cardinals $<\kappa$.
For a set of ordinals $C$, we write $\acc(C) := \{\alpha\in C \mid \sup(C \cap \alpha) = \alpha > 0\}$
and $\nacc(C) := C \setminus \acc(C)$.

$E^\kappa_\lambda$ denotes the set $\{\alpha < \kappa \mid \cf(\alpha) = \lambda\}$, and
$E^\kappa_{\geq \lambda}$, $E^\kappa_{<\lambda}$, $E^\kappa_{\neq\lambda}$, are defined analogously.
For a set $A$, we write $[A]^{\lambda}$ for $\{B\s A\mid |B|=\lambda\}$, and $[A]^{<\lambda}$ is defined analogously.
Finally, $\ch_\lambda$ asserts that $2^\lambda=\lambda^+$.

\subsection{Abstract trees}
\begin{defn}
A $\kappa$-tree $\mathbf T=(T,<_T)$ is said to be:
\begin{itemize}
\item \emph{full} iff for every limit $\alpha<\kappa$, there exists at most one $\alpha$-branch
that is vanishing;\footnote{A chain $X\s T$ is \emph{vanishing} iff there is no $y\in T$ such that $x\le_T y$ for all $x\in X$.}
\item \emph{Hausdorff} iff for every limit $\alpha<\kappa$ and all $x,y\in T_\alpha$,
if $x_\downarrow=y_\downarrow$, then $x=y$;
\item \emph{normal} iff for all ${\bar\alpha}<\alpha<\kappa$ and $x\in T_{\bar\alpha}$, there exists $y\in T_\alpha$ with $x<_T y$;
\item \emph{splitting} iff every node of $\mathbf T$ admits at least two  immediate successors;
\item \emph{$\kappa$-Aronszajn} iff $\mathbf T$ it has no $\kappa$-branches;
\item \emph{$\kappa$-Souslin} iff $\mathbf T$ it has no $\kappa$-branches and no $\kappa$-sized antichains.
\end{itemize}
\end{defn}

For a subset $E\s\kappa$, we let $T\restriction E:=\{ t\in T\mid \h(t)\in E\}$.

\begin{obs}\label{obs22}\begin{enumerate}
\item If a splitting full $\kappa$-tree exists, then $\kappa>2^{\aleph_0}$;
\item If a full $\kappa$-Aronszajn tree exists, then $\lambda^\theta<\kappa$ for all $\theta<\lambda<\kappa$;
\item A full $\kappa$-Aronszajn tree need not be normal;
\item Every full normal $\kappa$-Aronszajn tree is rigid on every club.
\end{enumerate}
\end{obs}
\begin{proof} (1) If $\mathbf T=(T,<_T)$ is a splitting $\kappa$-tree,
then $T\restriction\omega$ contains a copy of $({}^{<\omega}2,{\s})$.
If $\mathbf T$ is in addition full, then $2^{\aleph_0}\le|T_\omega|<\kappa$.

(2)  Suppose that $\mathbf T=(T,<_T)$ is a full $\kappa$-Aronszajn tree.
Let $\theta<\lambda$ be a pair of infinite cardinals below $\kappa$, and we shall show that $\lambda^\theta<\kappa$.
For every $x\in T$, denote $x^\uparrow:=\{ y\in T\mid x<_T y\}$.
As $\mathbf T$ is a $\kappa$-tree, the set $X:=\{ x\in T\mid |x^\uparrow|=\kappa\}$ has size $\kappa$,
and so $\mathbf X:=(X,<_T)$ is a normal $\kappa$-Aronszajn tree such that $X_\alpha\s T_\alpha$ for all $\alpha<\kappa$.
Now, by a standard fact, we may pick a sparse enough club $D\s\kappa$ such that for every pair $\bar\delta<\delta$ of ordinals from $D$,
for every $x\in X_{\bar\delta}$, there are at least $\lambda$-many extensions of $x$ in $X_\delta$.
Let $\langle \delta_i\mid i\le\theta\rangle$ be the increasing enumeration of the first $\theta$ elements of $D\setminus\{\min(D)\}$.
For all $i\le\theta$, if $\mathbf T$ has a vanishing $\delta_i$-branch, then it is unique and we denote it by $b_i$; otherwise, we just let $b_i$ denote an arbitrary $\delta_i$-branch.
As $\delta_0\in D\setminus\{\min(D)\}$, $|X_{\delta_0}|\ge\lambda>\theta$,
so we may pick $x\in X_{\delta_0}\setminus\bigcup_{i<\theta}b_i$.
Evidently, $x^\uparrow\cap\mathbf X$ has at least $\lambda^\theta$ many non-vanishing $\delta_\theta$-branches.
Thus, $\lambda^\theta\le|X_{\delta_\theta}|\le|T_{\delta_\theta}|<\kappa$.

(3)  Suppose that $\mathbf T^1=(T^1,<_1)$ is a full $\kappa$-Aronszajn tree
and that $\mathbf T^0=(T^0,<_0)$ is a full $\varkappa$-Aronszajn tree for some cardinal $\varkappa<\kappa$.
Define a $\kappa$-tree $\mathbf T=(T,<_T)$ by letting:
\begin{itemize}
\item $T:=(\{0\}\times T^0)\cup(\{1\}\times T^1)$, and
\item $(i,s)<_T(j,t)$ iff $i=j$ and $s<_i t$.
\end{itemize}

Then $\mathbf T$ is a full $\kappa$-Aronszajn tree that is not normal.

(4) Suppose that $\mathbf T=(T,<_T)$ is a full normal $\kappa$-Aronszajn tree.
Towards a contradiction, suppose that $\mathbf T$ is not rigid on every club. This means that we may fix a club $D\s\kappa$,
an automorphism $\pi:T\restriction D\rightarrow T\restriction D$ of the tree $(T\restriction D,<_T)$, and some node $x\in\dom(\pi)$ such that $\pi(x)\neq x$.

If, for every $\alpha\in\acc(D\setminus \h(x))$, every $\alpha$-branch to which $x$ belongs is not vanishing, then using normality of $\mathbf T$ we could have recursively constructed a $\kappa$-branch,
thus contradicting the fact that $\mathbf T$ is $\kappa$-Aronszajn. It follows that we may pick an $\alpha\in\acc(D\setminus \h(x))$ and a vanishing $\alpha$-branch $B$ to which $x$ belongs.
As $\pi(x)\neq x$, $B':=\{ t\in T\mid \exists s\in \pi[B]\,(y<_T s)\}$ is another $\alpha$-branch,
so since $\mathbf T$ is full, $B'$ is not vanishing, and we may let $t'\in T_\alpha$ be an upper bound for $B'$.
Then $\pi^{-1}(t')$ is an upper bound for $B$, contradicting the fact that it is vanishing.
\end{proof}

In reading the next proposition, recall that by \cite[Theorem~B]{paper24}, $\gch$ together with $\square(\aleph_2)$ imply that there exists an $\aleph_1$-complete $\aleph_2$-Souslin tree.

\begin{prop}\label{friedman} Assuming the consistency of a Mahlo cardinal,
it is consistent that $\gch$, $\diamondsuit(\aleph_1)$ and $\square(\aleph_2)$ all hold,
and yet there are no full $\aleph_2$-Aronszajn trees.
\end{prop}
\begin{proof}
In \cite[Theorem~7.1 of \S XI]{sh:f},
Shelah starts with any given Mahlo cardinal $\kappa$ in $\mathsf{L}$ and pass to a forcing extension
in which the $\gch$ holds, $\kappa=\aleph_2$,
and every stationary subset of $E^{\aleph_2}_{\omega}$ contains a closed copy of $\omega_1$.
As this is preserved by any forcing of size $\aleph_1$, we may moreover assume that $\diamondsuit(\aleph_1)$ holds in the extension.
By letting $\kappa$ be a non-weakly-compact cardinal (e.g., be the least Mahlo cardinal in $\mathsf{L}$),
\cite[Claim~1.10]{MR908147} ensures that $\square(\aleph_2)$ holds in the extension, as well.
Hereafter, work in this model, and suppose towards a contradiction that $\mathbf T=(T,<_T)$ is a full $\aleph_2$-Aronszajn tree.

Given $\delta\in  E^{\aleph_2}_\omega$,
if there is a vanishing $\delta$-branch then it is unique, and we denote it by $B_\delta$;
otherwise, we let $B_\delta$ be an arbitrary $\delta$-branch.
Let $\gamma<\aleph_2$.
For every $\delta\in E^{\aleph_2}_\omega$ above $\gamma$, there exists a unique $t\in T_\gamma\cap B_\delta$,
so since $|T_\gamma|<\aleph_2$, we may fix some $t_\gamma\in T_\gamma$ for which $\{ \delta\in E^{\aleph_2}_{\omega}\mid t_\gamma\in B_\delta\}$ is stationary.

Next, as in the proof of Clause~(2) of Observation~\ref{obs22},
fix $X\s T$ for which $\mathbf X:=(X,<_T)$ is a normal $\aleph_2$-Aronszajn tree such that $X_\alpha\s T_\alpha$ for all $\alpha<\aleph_2$,
and fix a club $D\s\aleph_2$ such that for every pair $\gamma<\delta$ of ordinals from $D$,
for every node $x\in X_\gamma$, there are two distinct extensions of $x$ in $X_\delta$.
Let $\gamma$ denote the second element of $D$, so that $|X_\gamma|>1$.
By the choice of $t_\gamma$, the set $S:=\{ \delta\in E^{\aleph_2}_{\omega}\cap D\mid t_\gamma\in B_\delta\}$ is stationary,
so it contains a closed copy of $\omega_1$.

Fix a strictly increasing and continuous map $\pi:\omega_1\rightarrow S$.
Pick $x\in X_\gamma\setminus\{ t_\gamma\}$. As $x\notin B_{\pi(i)}$ for every $i<\omega_1$,
it is now easy to construct a system of nodes $\langle x^f_i \mid f\in{}^{\omega_1}2, i<\omega_1\rangle$ in such a way that
for all $f,g\in{}^{\omega_1}2$:
\begin{itemize}
\item for all $i<j<\omega_1$, $x^f_i\in X_{\pi(i)}$ and $x<_T x^f_i <_T x^f_j$;
\item for every $i<\omega_1$, $x^f_i = x^g_i$ iff $f(i)=g(i)$.
\end{itemize}

Let $\tau:=\sup(\im(\pi))$.
Then the above system induces $2^{\aleph_1}$ many distinct $\tau$-branches through $\mathbf T$,
and since $\mathbf T$ is full, this must mean that $|T_\tau|\ge2^{\aleph_1}$, contradicting the fact that $|T_\tau|<\aleph_2$.
\end{proof}
\begin{defn} A weak embedding
from a tree $\mathbf S=(S,<_S)$ to a tree $\mathbf T=(T,<_T)$
is a map $f:S\rightarrow T$ satisfying that all $s,s'\in S$ with $s<_S s'$, $f(s)<_T f(s')$.
\end{defn}
Note that a weak embedding may be constant on antichains.

\begin{defn}\label{def25}
For a sequence of $\kappa$-trees $\langle \mathbf{T}^i \mid i<\tau \rangle$
with $\mathbf T^i = (T^i, {<_{T^i}})$ for each $i<\tau$,
the product
$\bigotimes_{i<\tau} \mathbf{T}^i$
is defined to be the tree $\mathbf T=({T}, {<_{{T}}})$,
where:
\begin{itemize}
\item $T=\bigcup\{\prod_{i<\tau}(T^i)_\alpha\mid \alpha<\kappa\}$;
\item $\vec{s} <_{{T}} \vec{t}$ iff $\vec s(i) <_{T^i} \vec t(i)$ for every $i<\tau$.
\end{itemize}
\end{defn}

\begin{prop}\label{weakpr}
Suppose that:
\begin{itemize}
\item $\mathbf S=(S,<_S)$ and $\mathbf T=(T,<_T)$ are $\kappa$-trees;
\item $\mathbf S\otimes\mathbf T$ is a $\kappa$-Souslin tree.
\end{itemize}

Then there are no weak embeddings from $\mathbf S$ to $\mathbf T$.
\end{prop}
\begin{proof} Note that the second bullet implies that $\mathbf S$ and $\mathbf T$ are $\kappa$-Souslin trees.
Towards a contradiction, suppose that $f:S\rightarrow T$ is a weak embedding from $\mathbf S$ to $\mathbf T$.
\begin{claim}\label{claim21} For every $s\in S$, $\h_{\mathbf T}(f(s))\ge\h_{\mathbf S}(s)$.
\end{claim}
\begin{proof} This is clear.
\end{proof}
As $\mathbf S$ is $\kappa$-Souslin,
the set $S':=\{ s\in S\mid s\text{ admits }\kappa\text{-many extensions in }S\}$ is co-bounded in $S$.
Put $X:=f[S']$ and
note it is of size $\kappa$.
Indeed, otherwise, there exists $A\in[S']^\kappa$ such that $|f[A]|=1$. As $\mathbf S$ is $\kappa$-Souslin,
we must be able to find a pair $s<_S s'$ in $A$, but then $f(s)<_T f(s')$, contradicting the fact that $|f[\{s,s'\}]|=|f[A]|=1$.

\begin{claim} The following set has size $\kappa$:
$$Y:=\{ x\in X\mid x\text{ admits two incompatible proper extensions in }X\}.$$
\end{claim}
\begin{proof} Suppose not. It follows that $\epsilon:=\sup\{\h_{\mathbf T}(x)\mid x\in Y\}$ is smaller than $\kappa$.
By the definition of $\epsilon$, for every $x\in X$ with $\h_{\mathbf T}(x)>\epsilon$, $U_x:=\{ y\in X\mid x<_T y\}$ is linearly ordered by $<_T$.
As $\mathbf T$ is $\kappa$-Aronszajn, the chain $U_x$ cannot have size $\kappa$.
As $\mathbf T$ is narrow, we may fix a club $D\s\kappa\setminus(\epsilon+1)$ such that, for every $\delta\in D$, for every $x\in X$ with $\epsilon<\h_{\mathbf T}(x)<\delta$,
$\sup\{ \h_{\mathbf T}(y)\mid y\in U_x\}<\delta$. Recalling that $|X|=\kappa$ and by possibly shrinking $D$ further, we may assume that for each $\gamma\in D$,
there exists $x_\gamma\in X$ with $\gamma<\h_{\mathbf T}(x_\gamma)<\min(D\setminus(\gamma+1))$.
But, then for every pair $\gamma<\delta$ of ordinals from $D$, $x_\delta\notin U_{x_\gamma}$, meaning that $\{ x_\gamma\mid \gamma\in D\}$ is a $\kappa$-sized antichain in $\mathbf T$,
contradicting the fact that $\mathbf T$ is $\kappa$-Souslin.
\end{proof}

\begin{claim} For every $y\in Y$, there is a pair  $(s_y^0,s_y^1)\in S\times S$ such that:
\begin{itemize}
\item $f(s_y^0)$ and $f(s_y^1)$ are incompatible proper extensions of $y$;
\item $\h_{\mathbf S}(s_y^0)=\h_{\mathbf T}(f(s_y^1))$.
\end{itemize}
\end{claim}
\begin{proof} Let $y\in Y$ and then pick $x_0,x_1\in X$ that are two incompatible proper extensions of $y$.
Without loss of generality, $\h_{\mathbf T}(x_0)\le \h_{\mathbf T}(x_1)$. For each $i<2$, as $x_i\in X$, we may pick $s_i\in S'$ such that $f(s_i)=x_i$.
By Claim~\ref{claim21}, $\h_{\mathbf S}(s_0)\le h_{\mathbf T}(f(s_0))\le h_{\mathbf T}(f(s_1))$,
so since $s_0\in S'$, we may pick $s_y^0\in S$ with $s_0\le_S s_y^0$ such that $\h_{\mathbf S}(s_y^0)=\h_{\mathbf T}(x_1)$.
As $f$ is a weak embedding, $y<_T f(s_0)\le_T f(s_y^0)$.
So, letting $s_y^1:=s_1$, it is the case that
$f(s_y^0)$ and $f(s_y^1)$ are incompatible proper extensions of $y$
and that $\h_{\mathbf S}(s_y^0)=\h_{\mathbf T}(x_1)=\h_{\mathbf T}(f(s_y^1))$.
\end{proof}

It follows that we may fix a sparse enough set $\Gamma\in[\kappa]^\kappa$,
and a sequence $\langle (y_\gamma,s_\gamma^0,s_\gamma^1)\mid \gamma\in\Gamma\rangle$ consisting of elements in $Y\times S\times S$ such that for every pair $\gamma<\delta$ of elements of $\Gamma$:
\begin{enumerate}
\item $f(s_\gamma^0)$ and $f(s_\gamma^1)$ are incompatible proper extensions of $y_\gamma$;
\item $\gamma=\h_{\mathbf T}(y_\gamma)<\h_{\mathbf T}(f(s_\gamma^1))=\h_{\mathbf S}(s_\gamma^0)\le \h_{\mathbf T}(f(s_\gamma^0))<\delta$.
\end{enumerate}

As $\{ (s^0_\gamma,f(s^1_\gamma))\mid \gamma\in\Gamma\}$ is a $\kappa$-sized subset of the $\kappa$-Souslin tree $\mathbf S\otimes\mathbf T$,
we may fix a pair $\gamma<\delta$ of elements of $\Gamma$ such that $s_\gamma^0<_Ss_\delta^0$ and $f(s_\gamma^1)<_Tf(s_\delta^1)$.
Since $f$ is a weak embedding, $s_\gamma^0<_S s_\delta^0$ implies $f(s_\gamma^0)<_T f(s_\delta^0)$.
So, for every $i<2$, combining the facts that $f(s_\gamma^i)<_T f(s_\delta^i)$ and $y_\delta<_T f(s_\delta^i)$ with Clause~(2), we infer that $f(s_\gamma^i)<_T y_\delta$.
Then $f(s_\gamma^0),f(s_\gamma^i)<_T y_\delta$, contradicting Clause~(1)
\end{proof}
\subsection{Streamlined trees}
\begin{defn}[\cite{paper23}]
A \emph{streamlined $\kappa$-tree} is a subset $T\s{}^{<\kappa}H_\kappa$
such that the following two conditions are satisfied:
\begin{enumerate}
\item $T$ is downward-closed, i.e, for every $t\in T$, $\{ t\restriction \alpha\mid \alpha<\kappa\}\s T$;
\item for every $\alpha<\kappa$, the set
$T_\alpha:=T\cap{}^\alpha\kappa$ is nonempty and has size $<\kappa$.
\end{enumerate}
For every $\alpha\le\kappa$, we denote
$\mathcal B(T\restriction\alpha):=\{f\in{}^\alpha H_\kappa\mid\forall\beta<\alpha\,(f\restriction\beta\in T)\}$.
\end{defn}

By convention, we identify a streamlined tree $T$ with the poset $\mathbf T=(T,{\s})$.
Note that every streamlined $\kappa$-tree $T$ is Hausdorff,
and that it is full iff $|\mathcal B(T\restriction\alpha)\setminus T_\alpha|\le 1$ for every $\alpha\in\acc(\kappa)$.

\begin{defn}
A streamlined tree $T\s{}^{<\kappa}H_\kappa$ is said to be:
\begin{itemize}
\item \emph{binary} iff $T\s{}^{<\kappa}2$;
\item \emph{prolific} iff for all $\alpha<\kappa$ and $t\in T_\alpha$,
$\{ t{}^\smallfrown\langle i\rangle\mid i<\max\{\omega,\alpha\}\}\s T$.
\end{itemize}
\end{defn}

Note that every prolific tree is splitting.

\subsection{Coherent sequences}
Recall that a \emph{$C$-sequence over $\kappa$} is a sequence $\vec C=\langle C_\alpha\mid\alpha<\kappa\rangle$ such that,
for every $\alpha<\kappa$, $C_\alpha$ is a closed subset of $\alpha$ with $\sup(C_\alpha)=\sup(\alpha)$.
It is \emph{$\lambda$-bounded} iff $\otp(C_\alpha)\le\lambda$ for all $\alpha<\kappa$.
For a binary relation $\mathcal R$ over $[\kappa]^{<\kappa}$, $\vec C$ is said to be \emph{$\mathcal R$-coherent}
iff for all $\beta<\alpha<\kappa$ such that $\beta\in\acc(C_\alpha)$, it is the case that $C_\beta\mathrel{\mathcal R}C_\alpha$.
In this paper, we shall only be concerned with the binary relations $\sq,\sq^*$ and $\sq_\lambda$. They are defined as follows:
\begin{itemize}
\item $D\sq C$ iff there exists some ordinal $\beta$ such that $D = C \cap \beta$;
\item $D\sq^* C$ iff $D \setminus \varepsilon\sqsubseteq C \setminus \varepsilon$
for some $\varepsilon < \sup(D)$;
\item $D \sq_\lambda C$ iff $D \sq C$ or ($\otp(C)<\lambda$ and $\nacc(C)$ consists only of successor ordinals).
\end{itemize}

\begin{defn}[Jensen's and Baumgartner's squares] For an infinite cardinal $\lambda$:
\begin{itemize}
\item $\square_\lambda$ asserts there is a $\sq$-coherent $\lambda$-bounded $C$-sequence over $\lambda^+$;
\item $\square^B_\lambda$ asserts there is a $\sq_\lambda$-coherent $\lambda$-bounded $C$-sequence over $\lambda^+$.
\end{itemize}
\end{defn}

\begin{defn}[Special case of the proxy principle from \cite{paper22}]\label{proxydef}
Suppose that $\theta\le\kappa$ is a cardinal,
$\mathcal R$ is a binary relation over $[\kappa]^{<\kappa}$ and $\mathcal S$ is a collection of stationary subsets of $\kappa$.

The principle $\p^-(\kappa,2,\mathcal{R},\theta,\mathcal{S})$ asserts the existence of an $\mathcal R$-coherent $C$-sequence $\vec C=\langle C_\alpha\mid \alpha<\kappa\rangle$ possessing the following `guessing' feature.
For every sequence $\langle B_i\mid i<\theta\rangle$ of cofinal subsets of $\kappa$,
for every $S\in\mathcal{S}$, there are stationarily many $\alpha\in S$
such that $\sup(\nacc(C_\alpha)\cap B_i)=\alpha$ for all $i<\min\{\alpha,\theta\}$.

The principle $\p(\kappa, 2, \mathcal R, \theta, \mathcal S)$
asserts that $\p^-(\kappa, 2, \mathcal R, \theta, \mathcal S)$
and $\diamondsuit(\kappa)$ both hold.
\end{defn}
\begin{conv} If we omit $\mathcal S$, then we mean $\mathcal S:=\{\kappa\}$.
\end{conv}

\begin{defn}[\cite{paper22}]\label{xbox} $\boxtimes^-(\kappa)$ stands for $\p^-(\kappa,2,{\sq},1)$.
\end{defn}

By \cite[\S3]{lh_trees_squares_reflection},
the generic for the forcing to add a $\sq$-coherent $C$-sequence over $\kappa$ by initial segments will constitute a $\boxtimes^-(\kappa)$-sequence.
So the principle of Definition~\ref{proxydef} should be understood as asserting the existence of an $\mathcal R$-coherent $C$-sequence over $\kappa$ possessing some generic properties.

\section{Diamond-type prediction principles}\label{secdia}

We open this section by recalling Jensen's diamond principle $\diamondsuit(\kappa)$ and two of its equivalent forms.

\begin{fact}[{\cite[Lemma~2.2]{paper22}}]\label{def_Diamond_H_kappa} The following are equivalent:
\begin{enumerate}
\item $\diamondsuit(\kappa)$, i.e., there is a sequence $\langle f_\beta\mid \beta<\kappa\rangle$ such that
for every function $f:\kappa\rightarrow\kappa$, the set $\{ \beta<\kappa\mid f\restriction\beta=f_\beta\}$ is stationary in $\kappa$.
\item $\diamondsuit^-(H_\kappa)$, i.e., there is a sequence $\langle \Omega_\beta \mid \beta < \kappa \rangle$ such that for all $p\in H_{\kappa^{+}}$  and $\Omega \subseteq H_\kappa$,
there exists an elementary submodel $\mathcal M\prec H_{\kappa^{+}}$ such that:
\begin{itemize}
\item $p\in\mathcal M$;
\item $\mathcal M\cap\kappa\in\kappa$;
\item $\mathcal M\cap \Omega=\Omega_{\mathcal M\cap\kappa}$.
\end{itemize}
\item $\diamondsuit(H_\kappa)$, i.e., there are a partition $\langle R_i \mid  i < \kappa \rangle$ of $\kappa$
and a sequence $\langle \Omega_\beta \mid \beta < \kappa \rangle$ such that for all $p\in H_{\kappa^{+}}$, $\Omega \subseteq H_\kappa$,
and $i<\kappa$,
there exists an elementary submodel $\mathcal M\prec H_{\kappa^{+}}$ such that:
\begin{itemize}
\item $p\in\mathcal M$;
\item $\mathcal M\cap\kappa\in  R_i$;
\item $\mathcal M\cap \Omega=\Omega_{\mathcal M\cap\kappa}$.
\end{itemize}
\end{enumerate}
\end{fact}

We now introduce two new diamond-type principles.

\begin{defn}[Diamonds for trees]\label{def41}
Suppose that $S\s E^\kappa_{>\omega}$ and $B\s\kappa$ are stationary sets.
A sequence $\langle f_\beta\mid \beta\in B\rangle$ is said to witness:
\begin{itemize}
\item $\diamondsuit_{S,B}(\kappa\textup{-trees})$
iff for every streamlined $\kappa$-tree $T$, there are stationarily many $\alpha\in S$ such that,
for every $f\in T_\alpha$, $\{ \beta\in B\cap\alpha\mid f\restriction\beta=f_\beta\}$ is stationary in $\alpha$;
\item $\diamondsuit^*_{S,B}(\kappa\textup{-trees})$
iff for every streamlined $\kappa$-tree $T$, there exists a club $D\s\kappa$ such that,
for every $\alpha\in S\cap D$, for every $f\in T_\alpha$, $\{ \beta\in B\cap\alpha\mid f\restriction\beta=f_\beta\}$ is stationary in $\alpha$.
\end{itemize}
\end{defn}
\begin{conv} If we omit $B$, then we mean $B:=\kappa$.
\end{conv}

\begin{remark}\label{jen}
\begin{enumerate}
\item  If $\kappa$ is weakly compact and $\vec f=\langle f_\beta\mid\beta<\kappa\rangle$ witnesses $\diamondsuit(\kappa)$ as in Fact~\ref{def_Diamond_H_kappa}(1),
then the set $S:=\{ \alpha\in \reg(\kappa)\setminus\{\omega\}\mid \vec f\restriction\alpha\text{ witnesses }\diamondsuit(\alpha)\}$ is stationary in $\kappa$,
and hence $\diamondsuit^*_S(\kappa\textup{-trees})$ holds.
\item Jensen's construction in $\mathsf{L}$ (see \cite[Lemma~6.5]{jen72}) of a $\diamondsuit(\kappa)$-sequence $\vec f=\langle f_\beta\mid\beta<\kappa\rangle$
has the property that $\vec f\restriction\alpha$ witnesses $\diamondsuit(\alpha)$ for every
regular uncountable $\alpha<\kappa$.
In particular, in $\mathsf{L}$, for every Mahlo cardinal $\kappa$,
$\diamondsuit^*_S(\kappa\textup{-trees})$ holds for $S:=\reg(\kappa)\setminus\{\omega\}$.
\item $\diamondsuit^*_{S,B}(\kappa\textup{-trees})$ implies the principle $\diamondsuit_B^S$ of \cite[\S2.2]{paper07}.
\item $\diamondsuit_{S,B}(\kappa\textup{-trees})$ implies the principle $\prod(B,\theta,S)$ of \cite{paper38} for all $\theta<\kappa$.
\item A simple variation of the proof of Proposition~\ref{friedman} establishes that if every stationary subset of $E^{\aleph_2}_{\omega}$ contains a closed copy of $\omega_1$,
then $\diamondsuit^*_{E^{\omega_2}_{\omega_1},E^{\omega_2}_{\omega_0}}(\aleph_2\textup{-trees})$ fails.
\end{enumerate}
\end{remark}

\begin{lemma}\label{dtree} Suppose that $S\s E^\kappa_{>\omega}$ and $B\s\kappa$ are stationary sets. Then:
\begin{enumerate}
\item $\diamondsuit^*_{S,B}(\kappa\textup{-trees})\implies\diamondsuit_{S,B}(\kappa\textup{-trees})\implies\diamondsuit(B)\implies\diamondsuit(\kappa)$;
\item If $\diamondsuit^*_S(\kappa\textup{-trees})$ holds, then $\{ \alpha\in S\mid \cf(|\alpha|)\neq|\alpha|\}$
is nonstationary;
\item If $\diamondsuit_S(\kappa\textup{-trees})$ holds, then $\kappa$ is either a Mahlo cardinal or the successor of a regular uncountable cardinal.
\end{enumerate}
\end{lemma}
\begin{proof} (1) The first and third implications are immediate.
For the second one, suppose that $\langle f_\beta\mid \beta\in B\rangle$ witnesses $\diamondsuit_{S,B}(\kappa\textup{-trees})$,
and we shall verify that it witnesses $\diamondsuit(B)$.
\begin{claim} For every $f\in{}^\kappa2$, the set $\{ \beta\in B\mid f\restriction\beta=f_\beta\}$
is stationary in $\kappa$ and it moreover reflects stationarily often in $S$.
\end{claim}
\begin{proof} Given $f\in{}^\kappa2$,
the set $T:=\{ f\restriction\alpha\mid \alpha<\kappa\}$ is a streamlined $\kappa$-tree.
Therefore, the following set is stationary in $\kappa$:
$$S':=\{\alpha\in S\mid \forall t\in T_\alpha\,(\{ \beta\in B\cap\alpha\mid t\restriction\beta=f_\beta\}\text{ is stationary in }\alpha)\}.$$
Denote $B':=\{\beta\in B\mid f\restriction\beta=f_\beta\}$.
Evidently,
$$\{\alpha\in S\mid B'\cap\alpha\text{ is stationary in }\alpha\}=S',$$
so are done.
\end{proof}

(2) Suppose that $\langle f_\beta\mid \beta<\kappa\rangle$ witnesses $\diamondsuit^*_S(\kappa\textup{-trees})$.
Evidently, $T:=\{ f\in{}^{<\kappa}2\mid f^{-1}\{1\}\text{ is finite}\}$
is a streamlined $\kappa$-tree, so let us fix a club $D\s\kappa$ such that,
for every $\alpha\in S\cap D$, for every $f\in T_\alpha$, $\{ \beta<\alpha\mid f\restriction\beta=f_\beta\}$ is stationary in $\alpha$.

If $\kappa$ is a limit cardinal, then let $C:=\{\alpha<\kappa\mid |\alpha|=\alpha\}$,
and if $\kappa$ is a successor cardinal, say, $\kappa=\lambda^+$, then let $C:=\kappa\setminus\lambda$.
In both cases, $C$ is a club in $\kappa$.
We claim that $S\cap D\cap C\s\{ \alpha<\kappa\mid \cf(|\alpha|)=|\alpha|\}$.

Towards a contradiction, suppose that $\alpha\in S\cap D\cap C$, and yet $|\alpha|$ is a singular cardinal.
Denote $\theta:=\cf(|\alpha|)$. As $\alpha\in C$, we may fix a club $C_\alpha$ in $\alpha$ of order-type $\theta$ such that $\min(C_\alpha)=\theta^+$.
Now, the collection $\mathcal F:=\{ f\in T_\alpha\mid f^{-1}\{1\}\s\theta^+\}$ has size $\theta^+$,
and for each $f\in\mathcal F$,
$G(f):=\{\beta\in C_\alpha\mid f\restriction\beta=f_\beta\}$ is stationary in $\alpha$.
It follows that the map $f\mapsto\min(G(f))$ forms an injection from $\mathcal F$ to $C_\alpha$. This is a contradiction.

(3) The proof of Clause~(2) makes it clear that if $\diamondsuit_S(\kappa\textup{-trees})$ holds, then $\{ \alpha\in E^\kappa_{>\omega}\mid \cf(|\alpha|)=|\alpha|\}$ must be stationary in $\kappa$.
\end{proof}

\begin{defn}[Jensen-Kunen, \cite{jensen1969some}] A cardinal $\kappa$ is \emph{subtle} iff for every sequence $\langle A_\beta\mid\beta\in D\rangle$ over a club $D\s\kappa$,
there is a pair $\beta<\alpha$ of ordinals in $D$ such that $A_\beta\cap\beta=A_\alpha\cap\beta$.
\end{defn}
\begin{prop}\label{prop45} If $\kappa$ is a subtle cardinal,
then it is strongly inaccessible and there exists a stationary $S\s\reg(\kappa)\setminus\{\omega\}$ such that $\diamondsuit^*_S(\kappa\textup{-trees})$ holds.
\end{prop}
\begin{proof} By our convention, $\kappa$ is regular, though it is anyway easy to show that a subtle cardinal must be regular.
Also, if there exists a cardinal $\lambda<\kappa$ such that $2^\lambda\ge\kappa$,
then by taking an injective sequence $\langle A_\beta\mid\beta<\kappa\rangle$ of cofinal subsets of $\lambda$,
there is no pair $\beta<\alpha$ of ordinals in the club $\kappa\setminus\lambda$ such that $A_\beta=A_\alpha\cap\beta$.
So $\kappa$ is a strong limit.

Now, for the diamond, we just run the same proof of Kunen's theorem that the usual $\diamondsuit$ holds at a subtle cardinal.
Suppose that $\kappa$ is subtle. Define a sequence of pairs $\sigma=\langle (X_\alpha,C_\alpha)\mid \alpha<\kappa\rangle$ by recursion on $\alpha$.
For every $\alpha<\kappa$ such that $\sigma\restriction\alpha$ has already been defined, there are two options:

$\br$ If there exists some $Y\s\alpha$ such that $\{ \beta<\alpha\mid Y\cap\beta=X_\beta\}$ is nonstationary,
then pick a set $X_\alpha\s\alpha$ and a club $C_\alpha\s\alpha$ such that $\{ \beta<\alpha\mid X_\alpha\cap\beta=X_\beta\}\cap C_\alpha=\emptyset$,
and $\otp(C_\alpha)=\cf(\alpha)$.

$\br$ Otherwise, let $X_\alpha=C_\alpha$ be some club in $\alpha$ of order-type $\cf(\alpha)$.

\begin{claim} There exists a stationary $S\s E^\kappa_{>\omega}$ such that, for every $\alpha\in S$ and every $Y\s\alpha$,
$\{ \beta<\alpha\mid X_\beta=Y\cap\beta\}$ is stationary in $\alpha$.
\end{claim}
\begin{proof} Suppose not. Fix a club $D\s\acc(\kappa)$ such that, for every $\alpha\in D$,
either $\cf(\alpha)\le\omega$ or,
for some $Y_\alpha\s\alpha$, the set $\{ \beta<\alpha\mid X_\beta=Y_\alpha\cap\beta\}$ is nonstationary in $\alpha$.
Now, since $\kappa$ is subtle, we may find a pair $\beta<\alpha$ of ordinals in $D$ such that $X_\beta=X_\alpha\cap\beta$ and $C_\beta=C_\alpha\cap\beta$.
As $\cf(\alpha)=\otp(C_\alpha)>\otp(C_\alpha\cap\beta)=\otp(C_\beta)=\cf(\beta)$, we infer that $\cf(\alpha)>\omega$.
As $\alpha\in D\cap E^\kappa_{>\omega}$, the set $Y_\alpha$ witnesses that $X_\alpha$ and $C_\alpha$ were chosen in such a way that
$\{ \beta<\alpha\mid X_\alpha\cap\beta=X_\beta\}\cap C_\alpha=\emptyset$. However, from $C_\beta\s C_\alpha$, we infer that $\beta\in\acc(C_\alpha)\s C_\alpha$,
contradicting the fact that $X_\alpha\cap\beta=X_\beta$.
\end{proof}

Let $S$ be given by the preceding claim.
Fix a bijection $\pi:\kappa\leftrightarrow H_\kappa$ and set $f_\beta:=\bigcup\pi[X_\beta]$ for all $\beta<\kappa$.
\begin{claim} $\langle f_\beta\mid\beta<\kappa\rangle$ witnesses that $\diamondsuit^*_S(\kappa\textup{-trees})$ holds.
\end{claim}
\begin{proof}  Given a streamlined $\kappa$-tree $T$, consider the club $C:=\{\beta<\kappa\mid T\cap\pi[\beta]=T\restriction\beta\}$.
We claim that the club $D:=\acc(C)$ is as sought.
To this end, let $\alpha\in S\cap D$ and $f\in T_\alpha$.
As $\alpha$ is in particular in $C$,  $$Y:=\pi^{-1}[\{f\restriction\gamma\mid \gamma<\alpha\}]$$ is a subset of $\alpha$,
and hence $B:=\{ \beta\in C\cap\alpha\mid X_\beta=Y\cap\beta\}$ is in $\alpha$.
For every $\beta\in B$, it is the case that $X_\beta=Y\cap\beta=\pi^{-1}[\{f\restriction\gamma\mid \gamma<\beta\}]$,
and hence $f_\beta=\bigcup\pi[X_\beta]=\bigcup\{f\restriction\gamma\mid \gamma<\beta\}=f\restriction\beta$.
\end{proof}
By Lemma~\ref{dtree}(2) and by possibly shrinking $S$, we may assume that $S$ consists of regular cardinals.
Thus, $S\s\reg(\kappa)\setminus\{\omega\}$ and we are done.
\end{proof}

A famous weakening of Jensen's diamond principle is Ostaszewski's club principle.

\begin{defn}[Ostaszewski]\label{def-club}
For a subset $B\s\kappa$,
$\clubsuit(B)$ asserts the existence of a sequence
$\langle X_\beta\mid \beta\in B \rangle$ such that:
\begin{enumerate}
\item for every $\beta\in B$, $X_\beta\s\beta$ with $\sup(X_\beta)=\sup(\beta)$;
\item for every cofinal $X\s \kappa$, the set $\{ \beta \in B \mid X_\beta\s X \}$
is stationary in $\kappa$.
\end{enumerate}
\end{defn}

The next proposition establishes that $\diamondsuit^*_S(\kappa\textup{-trees})$ may hold at $\kappa$ a successor of a regular cardinal.
This includes cardinals as small as $\aleph_2$, as well as the successors of large cardinals.

\begin{prop}\label{cluba} Suppose that $\kappa=\lambda^+$ for a given regular uncountable $\lambda$,
and that  $\square^B_\lambda$ and $\ch_\lambda$ both hold.
Denote $S:=E^\kappa_\lambda$.
\begin{enumerate}
\item If $\clubsuit(\lambda)$ holds (e.g., if $\lambda$ is subtle), then so does $\diamondsuit^*_S(\kappa\textup{-trees})$;
\item If $\clubsuit(E^\lambda_\theta)$ holds for a given $\theta\in\reg(\lambda)$, then so does $\diamondsuit^*_{S,E^\kappa_\theta}(\kappa\textup{-trees})$.
\end{enumerate}
\end{prop}
\begin{proof} We settle for proving Clause~(2).
So, suppose that $\langle X_\beta\mid \beta\in E^\lambda_\theta\rangle$ witnesses $\clubsuit(E^\lambda_\theta)$,
for a given $\theta\in\reg(\lambda)$,
and we shall prove that $\diamondsuit^*_{S,B}(\kappa\textup{-trees})$ holds for $B:=E^\kappa_\theta$.
Using $\ch_\lambda$, fix a bijection $\pi:\kappa\leftrightarrow H_\kappa$.
Let $\langle C_\alpha\mid\alpha<\kappa\rangle$ be a $\lambda$-bounded $\sq_\lambda$-coherent $C$-sequence over $\kappa$.
By a standard argument (see \cite[Claim~3.20.1]{paper32}), we may fix a sequence of injections $\langle \varphi_\alpha:\alpha\rightarrow\lambda\mid\alpha<\kappa\rangle$
such that, for all $\beta<\alpha<\kappa$, if $C_\beta\sq C_\alpha$ and $\sup(C_\beta)=\beta$, then $\varphi_{\beta}=\varphi_\alpha\restriction\beta$.
Finally, for every $\beta\in B$, since it is the case that $\otp(C_\beta)\in E^\lambda_\theta$, we may let
\begin{itemize}
\item $I_\beta:=\{i<\beta\mid \varphi_\beta(i)\in X_{\otp(C_\beta)}\}$, and
\item $f_\beta:=\bigcup\{\pi(i)\mid i\in I_\beta\}$.
\end{itemize}

To see that $\langle f_\beta\mid \beta\in B\rangle$ forms a $\diamondsuit^*_{S,B}(\kappa\textup{-trees})$-sequence,
let $T$ be a given streamlined $\kappa$-tree.
Consider the club $$C:=\{ \gamma<\kappa\mid T\cap\pi[\gamma]=T\restriction\gamma\}.$$
We claim that the club $D:=\acc(C)$ is as required by Definition~\ref{def41}.
To this end, let  $\alpha\in S\cap D$ and $f\in T_\alpha$.
For each $\gamma\in C\cap\alpha$, it is the case that $\gamma\le\pi^{-1}(f\restriction\gamma)<\min(C\setminus(\gamma+1))$.
It thus follows that the following set is cofinal in $\alpha$:
$$Y:=\{ \pi^{-1}(f\restriction\gamma)\mid \gamma\in C\cap\alpha\}.$$
Furthermore, for every $i\in Y$, $$i=\pi^{-1}(f\restriction(\max(C\cap(i+1)))).$$

Since $\varphi_\alpha$ is an injection, by the Dushnik-Miller theorem, we may find a cofinal subset $Y'$ of $Y$ on which the map $i\mapsto \varphi_\alpha(i)$ is strictly increasing.
In particular, $\otp(Y')=\lambda$ and
the set $X:=\varphi_\alpha[Y']$ is cofinal in $\lambda$.
It follows that $e:=\{\eta\in E^\lambda_\theta\mid X_\eta\s X\}$ is stationary in $\lambda$,
and hence the following set is stationary in $\alpha$:
$$B^*:=\{ \beta\in B\cap\acc(C_\alpha)\mid \sup(Y'\cap\beta)=\beta\ \&\ \sup(\varphi_\alpha[Y'\cap\beta])=\otp(C_\alpha\cap\beta)\in e\}.$$
Let $\beta\in B^*$.
From $\otp(C_\alpha)=\lambda$,
we get that $C_\beta=C_\alpha\cap\beta$ and $\varphi_\beta=\varphi_\alpha\restriction\beta$.
Denote $Y_\beta:=Y'\cap\beta$, $\psi_\beta:=\varphi_\beta\restriction Y_\beta$, and $\eta:=\otp(C_\beta)$.
Then $\psi_\beta$ is a strictly increasing map from a cofinal subset of $\beta$ to a cofinal subset of $\eta$,
and $\eta\in e$.

For each $i\in I_\beta$, we have that $i<\beta$ and $\varphi_\beta(i)\in X_\eta\s X\s\varphi_\alpha[Y']$
and hence $I_\beta\s Y_\beta=\dom(\psi_\beta)$. So $\psi_\beta\restriction I_\beta$ is an order-preserving bijection from $I_\beta$ to $X_\eta$.
As $\sup(X_\eta)=\eta=\sup(\im(\psi_\beta))$, it follows that $\sup(I_\beta)=\sup(\dom(\psi_\beta))=\beta$.

As $I_\beta\s Y$ and $\sup(I_\beta)=\beta$, we altogether get that
\begin{align*}
f_\beta&=\bigcup\{\pi(i)\mid i\in I_\beta\}\\
&=\bigcup\{\pi(\pi^{-1}(f\restriction(\max(C\cap(i+1)))))\mid i\in I_\beta\}\\
&=\bigcup\{f\restriction(\max(C\cap(i+1)))\mid i\in I_\beta\},\\
&=f\restriction\beta,
\end{align*}
as sought.
\end{proof}

\begin{cor} For every uncountable cardinal $\mu$,
if $\square^B_{\mu^+}$, $\ch_\mu$ and $\ch_{\mu^+}$ all hold, then so does
$\diamondsuit^*_S(\mu^{++}\textup{-trees})$, with $S:=E^{\mu^{++}}_{\mu^+}$.
\end{cor}
\begin{proof} By the main result of \cite{Sh:922},
for every uncountable cardinal $\mu$, $\ch_\mu$ implies $\diamondsuit(\mu^+)$. Now, appeal to Proposition~\ref{cluba}(1) with $\lambda:=\mu^+$.
\end{proof}

We close this section by remarking that the proof of Proposition~\ref{cluba} works equally well with partial squares instead of $\square^B_\lambda$.
So, when combined with \cite[Theorem~2.1]{MR2833150},
we get that if
$\clubsuit(\lambda)$ and $\ch_\lambda$ both hold for a regular uncountable cardinal $\lambda$
such that $\kappa:=\lambda^+$ is not greatly Mahlo in $\mathsf{L}$,
then $\diamondsuit^*_S(\kappa\textup{-trees})$ holds
for some stationary $S\s E^\kappa_\lambda$.

\section{Full Souslin trees at strongly inaccessibles}\label{inacc}

\begin{thm}\label{thm43}
Suppose that:
\begin{itemize}
\item $\kappa$ is a strongly inaccessible cardinal;
\item $S\s E^\kappa_{>\omega}$ is stationary, and $\diamondsuit^*_S(\kappa\textup{-trees})$ holds;
\item  $\p^-(\kappa,2,{\sq^*},1,\{S\})$ holds.
\end{itemize}

Then there exists a streamlined, normal, prolific full $\kappa$-Souslin tree.
\end{thm}
\begin{proof}
Let $\langle f_\beta\mid\beta<\kappa\rangle$ be a witness for $\diamondsuit^*_S(\kappa\textup{-trees})$.
Using Lemma~\ref{dtree}(1) and Fact~\ref{def_Diamond_H_kappa}, we may also fix a sequence $\langle \Omega_\beta\mid\beta<\kappa\rangle$ witnessing $\diamondsuit^-(H_\kappa)$.
Fix a sequence $\vec C=\langle C_\alpha\mid \alpha<\kappa\rangle$ witnessing $\p^-(\kappa,2,{\sq^*},1,\{S\})$,
and fix a well-ordering $\vartriangleleft$ of $H_\kappa$.

Following the proof of \cite[Proposition~2.2]{paper26}, we shall recursively construct a sequence $\langle T_\alpha\mid \alpha<\kappa\rangle$ such that
$T:=\bigcup_{\alpha<\kappa}T_\alpha$ will constitute a normal prolific full streamlined $\kappa$-Souslin tree whose $\alpha^{\text{th}}$-level is $T_\alpha$.

Let $T_0:=\{\emptyset\}$, and for all $\alpha<\kappa$ let $$T_{\alpha+1}:=\{t{}^\smallfrown\langle i\rangle\mid t\in T_\alpha, i<\max\{\omega,\alpha\}\}.$$
Next, suppose that $\alpha\in\acc(\kappa)$ is such that $T\restriction\alpha$ has already been defined.
Constructing the level $T_\alpha$ involves deciding which branch through $T\restriction\alpha$ (if any) will \emph{not} have its limit placed into the (to-be-full) tree.
To ensure that the tree is normal, we attach to any node $x\in T\restriction C_\alpha$ some node $\mathbf{b}_x^\alpha\in\mathcal B(T\restriction\alpha)$ above it, and promise to satisfy
\begin{equation}\tag{$\star$}\label{promise}\{\mathbf{b}_x^\alpha\mid x\in T\restriction C_\alpha\}\s T_\alpha.\end{equation}

Let $x\in T\restriction C_\alpha$. We shall describe $\mathbf{b}_x^\alpha$ as the limit of a sequence $b_x^\alpha\in\prod_{\beta\in C_\alpha\setminus\dom(x)}T_\beta$ such that:
\begin{itemize}
\item $b_x^\alpha(\dom(x))=x$;
\item $b_x^\alpha(\beta')\subset b_x^\alpha(\beta)$ for every pair $\beta'<\beta$ of ordinals from $C_\alpha\setminus\dom(x)$;
\item $b_x^\alpha(\beta)=\bigcup\im(b_x^\alpha\restriction\beta)$ for all $\beta\in\acc(C_\alpha\setminus\dom(x))$.
\end{itemize}
The sequence is defined by recursion over $\beta\in C_\alpha\setminus\dom(x)$.
We start by letting $b_x^\alpha(\dom(x)):=x$. At successor step, for every $\beta\in C_\alpha\setminus(\dom(x)+1)$ such that $b_x^\alpha(\beta^-)$ has already been defined with $\beta^-:=\sup(C_\alpha\cap\beta)$,
we consult the following set:
$$Q^{\alpha, \beta}_x := \{ t\in T_\beta\mid \exists s\in \Omega_{\beta}[ (s\cup b^\alpha_x(\beta^-))\s t]\}.$$
Now, consider the two possibilities:
\begin{itemize}
\item If $Q^{\alpha,\beta}_x \neq \emptyset$, then let $b^\alpha_x(\beta)$ be its $\lhd$-least element;
\item Otherwise, let $b^\alpha_x(\beta)$ be the $\lhd$-least element of $T_\beta$ that extends $b^\alpha_x(\beta^-)$.
Such an element must exist, as the level $T_\beta$ was constructed so as to preserve normality.
\end{itemize}
Finally, for every $\beta\in\acc(C_\alpha\setminus\dom(x))$ such that $b_x^\alpha\restriction\beta$ has already been defined, we let $b_x^\alpha(\beta)=\bigcup\im(b_x^\alpha\restriction\beta)$.
By \eqref{promise} and the exact same proof of \cite[Claim~2.2.1]{paper26}, $b_x^\alpha(\beta)$ is indeed in $T_\beta$.

This completes the definition of $\mathbf{b}_x^\alpha$,
and it is clear that $\mathbf b_x^\alpha\in\mathcal B(T\restriction\alpha)$.
\begin{claim}\label{claim431} For every $t\in \{\mathbf{b}_x^\alpha\mid x\in T\restriction C_\alpha\}$, there exists a tail of $\varepsilon\in C_\alpha$ such that $t=\mathbf b^\alpha_{t\restriction\varepsilon}$.
\end{claim}
\begin{proof} Let $x\in T\restriction C_\alpha$ and write $t:=\mathbf{b}_x^\alpha$.
An inductive argument, utilizing the above canonical definition of $b_x^\alpha$ makes it clear that $x=t\restriction\dom(x)$ and that, furthermore,
$b_x^\alpha(\varepsilon)=t\restriction\varepsilon$ for every $\varepsilon\in C_\alpha\setminus\dom(x)$.
\end{proof}

Finally, we define $T_\alpha$ as follows:
$$T_\alpha:=\begin{cases}\mathcal B(T\restriction\alpha),&\text{if }f_\alpha=\mathbf{b}_x^\alpha\text{ for some }x\in T\restriction C_\alpha;\\
\mathcal B(T\restriction\alpha)\setminus\{f_\alpha\},&\text{otherwise}.\end{cases}$$

This completes our recursive construction of $\langle T_\alpha\mid\alpha<\kappa\rangle$. Now, let $T:=\bigcup_{\alpha<\kappa} T_\alpha$.
Since $\kappa$ is strongly inaccessible, the levels of $T$ have size $<\kappa$.
Altogether, $T$ is a normal, prolific, full streamlined $\kappa$-tree.
Fix a club $D\s\kappa$ such that,
for every $\alpha\in S\cap D$, for every $f\in T_\alpha$, $\{ \beta<\alpha\mid f\restriction\beta=f_\beta\}$ is stationary in $\alpha$.

\begin{claim}\label{claim432} Let $\alpha\in S\cap D$. Then $T_\alpha=\{\mathbf{b}_x^\alpha \mid x\in T\restriction C_\alpha\}$.
\end{claim}
\begin{proof} Let $\rho\in T_\alpha$, and we shall find some $x\in T\restriction C_\alpha$ such that $\rho=\mathbf b_x^\alpha$.
As $\alpha\in S\cap D$ and $\rho\in T_\alpha$, the following set is stationary in $\alpha$:
$$B_\rho=\{\beta\in\acc(C_\alpha) \mid \rho\restriction\beta=f_\beta\}.$$

Let $\beta\in B_\rho$. Since $f_\beta=\rho\restriction\beta$ is in $T$, the definition of $T_\beta$ above implies
that there exists some $x\in T\restriction C_\beta$ such that $f_\beta=\mathbf{b}_{x}^\beta$.
By Claim~\ref{claim431}, there exists a tail of $\varepsilon\in C_\beta$ such that $f_\beta=\mathbf{b}_{\rho\restriction\varepsilon}^\beta$.
By $\sq^*$-coherence of $\vec C$, we may then find a large enough $\varepsilon_\beta\in C_\beta$ such that
$f_\beta=\mathbf{b}_{\rho\restriction\varepsilon_\beta}^\beta$ and $C_\alpha\cap[\varepsilon_\beta,\beta)=C_\beta\cap[\varepsilon_\beta,\beta)$.
By Fodor's lemma for ordinals of uncountable cofinality, we may fix some $\varepsilon\in C_\alpha$ such that $B_\rho^\varepsilon:=\{ \beta\in B_\rho\mid \varepsilon_\beta\le\varepsilon\}$ is stationary.
Denote $x:=\rho\restriction\varepsilon$.
Then, for every $\beta\in B_\rho^\varepsilon$, it is the case that $\rho\restriction\beta=\mathbf b^\beta_{x}$. Furthermore,
since $C_\alpha\cap[\dom(x),\beta)=C_\beta\setminus\dom(x)$, it is the case that $\mathbf b^\beta_x=b^\alpha_{x}(\beta)$.
Altogether, $\rho=\mathbf b_x^\alpha$.
\end{proof}

Finally, since $T$ is splitting, to prove that $T$ is $\kappa$-Souslin, it suffices to prove that it has no antichains of size $\kappa$.
To this end, let $A$ be maximal antichain in $T$.
By \cite[Claim~2.2.2]{paper26},
the following set is stationary in $\kappa$:
$$B := \{ \beta <\kappa \mid  A\cap(T\restriction\beta)= \Omega_\beta\text{ is a maximal antichain in }T\restriction\beta \}.$$
As $\vec C$ witnesses $\p^-(\kappa,2,{\sq^*},1,\{S\})$,
we may now find some $\alpha\in S\cap D$ such that:
$$\sup(\nacc(C_\alpha)\cap B)=\alpha.$$
Using Claim~\ref{claim432}, the very same analysis of \cite[Claim~2.2.3]{paper26} implies that $A\s T\restriction\alpha$.
In particular, $|A|<\kappa$, as sought.
\end{proof}

Recalling Definition~\ref{xbox}, the following implies Theorem~\ref{thma}.
\begin{cor} Suppose that $\kappa$ is a subtle cardinal and $\p(\kappa,2,{\sq^*},1)$ holds.
Then there exists a streamlined, normal, prolific full $\kappa$-Souslin tree.
\end{cor}
\begin{proof} By Proposition~\ref{prop45}, $\kappa$ is strongly inaccessible and we may pick a stationary $S\s\reg(\kappa)\setminus\{\omega\}$ such that $\diamondsuit^*_S(\kappa\textup{-trees})$ holds.
By \cite[Lemma~3.8]{paper32}, $\p(\kappa,2,{\sq^*},1)$ implies $\p(\kappa,2,{\sq^*},1,\{S\})$.
Now, appeal to Theorem~\ref{thm43}.
\end{proof}

At the end of \cite{Sh:624}, it was announced that in $\mathsf{L}$, for every Mahlo cardinal $\kappa$,
there exists a full $\kappa$-Souslin tree.
The proof of Proposition~1.11 there reads ``Look carefully at the proof of 1.7.''
and the proof of the subsequent Corollary~1.12  makes use of a nonreflecting stationary set,
hence one clearly needs to add the restriction that $\kappa$ is not weakly compact.

Our proof of Theorem~\ref{thm43} is indeed the outcome of reading \cite{Sh:624} carefully.
Note, however, that our construction is necessarily different since it is based on the proxy principle,
which, by \cite[Theorem~1.12]{lh_trees_squares_reflection} is compatible with the assertion that all stationary subsets of $\kappa$ reflects.
\begin{cor}[{\cite[Corollary~1.12]{Sh:624}}] In $\mathsf{L}$, for every Mahlo cardinal $\kappa$ that is not weakly compact,
there exists a streamlined, normal, prolific full $\kappa$-Souslin tree.
\end{cor}
\begin{proof} Work in $\mathsf{L}$,
and suppose that $\kappa$ is a Mahlo cardinal that is not weakly compact.
By \cite[Theorem~3.12]{paper22}, $\p(\kappa,2,{\sq},1)$ holds.
By \cite[Lemma~3.8]{paper32}, the latter implies $\p(\kappa,2,{\sq^*},1,\{S\})$ for every stationary $S\s\kappa$.
Let $S:=\reg(\kappa)\setminus\{\omega\}$. By Remark~\ref{jen}, $\diamondsuit_S^*(\kappa\textup{-trees})$ holds.
Now, appeal to Theorem~\ref{thm43}.
\end{proof}

Our next task is obtaining a large family of full $\kappa$-Souslin trees.
This is achieved by strengthening the coherence relation from $\sq^*$ to $\sq$.

\begin{thm}\label{thm44} 	Suppose that:
\begin{itemize}
\item $\kappa$ is a strongly inaccessible cardinal;
\item $S\s E^\kappa_{>\omega}$ is stationary, and $\diamondsuit^*_S(\kappa\textup{-trees})$ holds;
\item  $\p^-(\kappa,2,{\sq},1,\{S\})$ holds.
\end{itemize}

Then there is a family $\mathcal T$ of $2^\kappa$ many streamlined, normal, binary, splitting, full $\kappa$-trees
such that $\bigotimes\mathcal T'$ is $\kappa$-Souslin for every nonempty $\mathcal T'\in[\mathcal T]^{<\kappa}$.
\end{thm}
\begin{proof}
Let $\langle f_\beta\mid\beta<\kappa\rangle$ be a witness for $\diamondsuit^*_S(\kappa\textup{-trees})$.
By Lemma~\ref{dtree}, we may assume that $S\s\reg(\kappa)$,
and by Fact~\ref{def_Diamond_H_kappa}, we may also fix sequences $\langle \Omega_\beta\mid\beta<\kappa\rangle$ and $\langle R_i\mid i<\kappa\rangle$
together witnessing $\diamondsuit(H_\kappa)$.
Since $S\s\reg(\kappa)$ and $\diamondsuit(\kappa)$ holds, by \cite[Lemma~3.8(2)]{paper28},
$\p^-(\kappa,2,{\sq},\kappa,\{S\})$ follows from $\p^-(\kappa,2,{\sq},1,\{S\})$,
and hence we may fix a sequence $\vec C=\langle C_\alpha\mid \alpha<\kappa\rangle$ witnessing $\p^-(\kappa,2,{\sq},\kappa,\{S\})$.
Without loss of generality, $0\in C_\alpha$ for all nonzero $\alpha<\kappa$.

Let $\pi:\kappa\rightarrow\kappa$ be such that $\alpha\in R_{\pi(\alpha)}$ for all $\alpha<\kappa$.
As $\kappa$ is strongly inaccessible, let $\lhd$ be some well-ordering of $H_\kappa$ of order-type $\kappa$,
and let $\phi:\kappa\leftrightarrow H_\kappa$ witness the isomorphism $(\kappa,\in)\cong(H_\kappa,\lhd)$.
Put $\psi:=\phi\circ\pi$.

We shall construct a sequence $\langle L^\eta\mid \eta\in {}^{<\kappa}2\rangle$ such that,
for all $\alpha<\kappa$ and $\eta\in {}^\alpha2$:
\begin{itemize}
\item[(i)] $L^\eta\s {}^\alpha2$;
\item[(ii)] for every $\beta<\alpha$, $L^{\eta\restriction\beta}=\{t\restriction\beta\mid t\in L^\eta\}$.
\end{itemize}

By convention, for every $\alpha\in\acc(\kappa+1)$ such that $\langle L^\eta\mid \eta\in {}^{<\alpha}2\rangle$ has already been defined,
and for every $\eta\in {}^\alpha2$, we shall let $T^\eta:=\bigcup_{\beta<\alpha}L^{\eta\restriction\beta}$,
so that $T^\eta$ is a tree of height $\alpha$ whose $\beta^{\text{th}}$ level is $L^{\eta\restriction\beta}$ for all $\beta<\alpha$.

\medskip

The construction of the sequence $\langle L^\eta\mid \eta\in {}^{<\kappa}2\rangle$ is by recursion on $\dom(\eta)$.
We start by letting $L^\emptyset:=\{\emptyset\}$.
For every $\alpha<\kappa$ such that $\langle L^\eta\mid \eta\in {}^\alpha2\rangle$ has already been defined,
for every $\eta\in {}^{\alpha+1}2$, let
$$L^\eta:=\{ t{}^\smallfrown\langle0\rangle,t{}^\smallfrown\langle1\rangle\mid t\in L^{\eta\restriction\alpha}\}.$$

Suppose now that $\alpha\in\acc(\kappa)$ is such that $\langle L^\eta\mid \eta\in {}^{<\alpha}2\rangle$ has already been defined.
We shall define a matrix
$$\mathbb B^\alpha=\langle b_x^{\alpha,\eta}\mid \beta\in C_\alpha, \eta\in {}^{\beta}2, x\in T^\eta\restriction C_\alpha\cap(\beta+1) \rangle$$
ensuring that $x\s b_x^{\alpha,\bar\eta}\s b_x^{\alpha,\eta}\in L^\eta$ whenever $\bar\eta\s\eta$.\footnote{\label{fnc}This also implies that the matrix is continuous,
i.e., for $\beta\in\acc(C_\alpha)$ $\eta\in {}^{\beta}2$ and $x\in T^\eta\restriction(C_\alpha\cap\beta)$, it is the case that $b_x^{\alpha,\eta}=\bigcup\{b_x^{\alpha,\eta\restriction\bar\beta}\mid \bar\beta\in C_\alpha\cap\beta\setminus\dom(x)\}$.}
Then, for all $\eta\in {}^\alpha2$ and $x\in T^\eta\restriction C_\alpha$,
it will follow that $\mathbf b_x^{\eta}:=\bigcup_{\beta\in C_\alpha\setminus\dom(x)}b_x^{\alpha,\eta\restriction\beta}$ is an element of $\mathcal B(T^\eta)$ extending $x$,
and we shall let
\begin{equation}\tag{$\star$}\label{promise2}L^\eta:=\begin{cases}\mathcal B(T^\eta),&\text{if }f_\alpha=\mathbf{b}_x^\eta\text{ for some }x\in T^\eta\restriction C_\alpha;\\
\mathcal B(T^\eta)\setminus\{f_\alpha\},&\text{otherwise}.\end{cases}\end{equation}

We now turn to define the components of the matrix $\mathbb B^\alpha$ by recursion on $\beta\in C_\alpha$.
So suppose that $\beta\in C_\alpha$ is such that
$$\mathbb B^\alpha_{<\beta}:=\langle b_x^{\alpha,\eta}\mid \bar\beta\in C_\alpha\cap\beta, \eta\in {}^{\bar\beta}2, x\in T^\eta\restriction C_\alpha\cap(\bar\beta+1) \rangle$$
has already been defined.

$\br$ For all $\eta\in {}^\beta2$ and $x\in T^\eta$ such that $\dom(x)=\beta$, let $b_x^{\alpha,\eta} := x$.

$\br$ For all $\eta\in {}^\beta2$ and $x\in T^\eta$ such that $\dom(x)<\beta$,
there are two main cases to consider:

$\br\br$ Suppose that $\beta\in \nacc(C_\alpha)$ and denote $\beta^-:=\sup(C_\alpha\cap\beta)$.

$\br\br\br$ If $\beta\in \acc(\kappa)$
and there exists a nonzero cardinal $\chi<\kappa$ such that all of the following hold:
\begin{enumerate}
\item There exists a sequence $\langle \eta_j\mid j<\chi\rangle$ of elements of ${}^\beta2$,
and a maximal antichain $A$ in the product tree $\bigotimes_{j<\chi}T^{\eta_j}$
such that
$\Omega_\beta=\{(\langle \eta_j\restriction\epsilon\mid j<\chi\rangle,A\cap{}^\chi(^\epsilon2))\mid \epsilon<\beta\}$;\footnote{As $\beta\in\acc(\kappa)$, it is the case that $\chi$,
$\langle \eta_j\mid j<\chi\rangle$ and $A$ are uniquely determined by $\Omega_\beta$.}
\item 				$\psi(\beta)$ is a sequence $\langle x_j\mid j<\chi\rangle$ such that $x_j\in T^{\eta_j\restriction\beta^-}\restriction(C_\alpha\cap\beta^-)$ for every $j<\chi$;
\item There exists a unique $j<\chi$ such that $\eta_j=\eta$ and $x_j=x$.
\end{enumerate}

In this case, by Clauses (1) and (2), the following set is nonempty
$$Q^{\alpha,\beta} := \{ \vec t\in \prod\nolimits_{j<\chi}L^{\eta_j}\mid \exists\vec s\in A\forall j<\chi[ (\vec s(j)\cup b^{\alpha,\eta_j\restriction\beta^-}_{x_j})\s\vec t(j)]\},$$
so we let $\vec t:=\min(Q^{\alpha,\beta},\lhd)$,
and then we let $b^{\alpha,\eta}_x:=\vec t(j)$ for the unique index $j$ of Clause~(3).
It follows that $b^{\alpha,\eta\restriction\beta^-}_{x}\s\vec t(j)= b^{\alpha,\eta}_x$.

$\br\br\br$ Otherwise, let $b_x^{\alpha,\eta}$ be the $\lhd$-least element of $L^{\eta}$ extending $b_x^{\alpha,\eta\restriction\beta^-}$.

$\br\br$ Suppose that $\beta\in\acc(C_\alpha)$. Then we define $b_x^{\alpha,\eta} := \bigcup\{b_x^{\alpha,\eta\restriction\bar\beta}\mid \bar\beta\in C_\alpha\cap\beta\setminus\dom(x)\}$.
We must show that the latter belongs to 	 $L^{\eta}$.
By \eqref{promise2}, it suffices to prove that $b_x^{\alpha,\eta}=\mathbf{b}_x^{\eta}$.
Since $\vec C$ is coherent and $\beta\in\acc(C_\alpha)$, it is the case that $C_\alpha\cap\beta=C_{\beta}$,
and hence proving $b_x^{\alpha,\eta}=\mathbf{b}_x^{\eta}$ amounts to showing that
$b_x^{\alpha,\eta\restriction\delta}=b_x^{\beta,\eta\restriction\delta}$ for all $\delta\in C_\beta\setminus\dom(x)$.
This is taken care of by the following claim.

\begin{claim}\label{cl441} $\mathbb B^\alpha_{<\beta}=\mathbb B^\beta$. That is, the following matrices coincide:
\begin{itemize}
\item $\langle b_y^{\alpha,\xi}\mid \bar\beta\in C_\alpha\cap\beta, \xi\in {}^{\bar\beta}2, y\in T^\xi\restriction C_\alpha\cap(\bar\beta+1) \rangle$;
\item $\langle b_y^{\beta,\xi}\mid \bar\beta\in C_\beta, \xi\in {}^{\bar\beta}2, y\in T^\xi\restriction C_\beta\cap(\bar\beta+1) \rangle$.
\end{itemize}
\end{claim}
\begin{proof} We already pointed out that $C_\alpha\cap\beta=C_{\beta}$, which for the scope of this proof we denote by $D$.
Now, by induction on $\delta\in D$, we prove that
$$\langle b_y^{\alpha,\xi}\mid \xi\in {}^{\delta}2, y\in T^\xi\restriction D\cap(\delta+1) \rangle=\langle b_y^{\beta,\xi}\mid \xi\in {}^{\delta}2, y\in T^\xi\restriction D\cap(\delta+1) \rangle.$$

The base case $\delta=\min(D)=0$ is immediate since $b_\emptyset^{\alpha,\emptyset}=\emptyset=b_\emptyset^{\beta,\emptyset}$.
The limit case $\delta\in\acc(D)$ follows from the continuity of the matrices under discussion as remarked in Footnote~\ref{fnc},
with the exception of those $y$'s such that $\dom(y)=\delta$,
but in this case, $b_y^{\alpha,\xi}=y=b_y^{\beta,\xi}$ for all $\xi\in{}^\delta2$.

Finally, assuming that $\delta^-<\delta$ are two successive elements of $D$ such that
$$\langle b_y^{\alpha,\xi}\mid \xi\in {}^{{\delta^-}}2, y\in T^\xi\restriction D\cap({\delta^-}+1) \rangle=\langle b_y^{\beta,\xi}\mid \xi\in {}^{{\delta^-}}2, y\in T^\xi\restriction D\cap({\delta^-}+1) \rangle,$$
we argue as follows. Given $\zeta\in{}^\delta2$ and $z\in T^\zeta\restriction D\cap(\delta+1)$, there are a few possible options.
If $\dom(z)=\delta$, then $b_z^{\alpha,\zeta}=z=b_z^{\beta,\zeta}$, and we are done.
If $\dom(z)<\delta$, then $\dom(z)\le\delta^-$ and, by the above construction,
for every $\gamma\in\{\alpha,\beta\}$, the value of
$b_z^{\gamma,\zeta}$ is completely determined by $\delta$, $\langle L^\xi\mid\xi\in {}^{\le\delta}2\rangle$, $\Omega_\delta$, $D$,
$\psi(\delta)$, $\zeta$, $x$, and
$\langle b_y^{\gamma,\xi}\mid \xi\in {}^{\delta^-}2, y\in T^\xi\restriction (D\cap\delta^-)\rangle$
in such a way that our inductive assumptions imply that $b_z^{\alpha,\zeta}=b_z^{\beta,\zeta}$.
\end{proof}

This completes the definition of the matrix $\mathbb B^\alpha$,
from which we derive $\mathbf b_x^{\eta}:=\bigcup_{\beta\in C_\alpha\setminus\dom(x)}b_x^{\alpha,\eta\restriction\beta}$ for all $\eta\in{}^\alpha2$ and $x\in T^\eta\restriction C_\alpha$,
and then we define $L^\eta$ as per \eqref{promise2}.

\begin{claim}\label{c441} For all $\eta\in{}^\alpha2$ and $t\in \{\mathbf{b}_x^\eta\mid x\in T^\eta\restriction C_\alpha\}$, there exists a tail of $\varepsilon\in C_\alpha$ such that $t=\mathbf b^\eta_{t\restriction\varepsilon}$.
\end{claim}
\begin{proof} This follows from the canonical nature of the construction,
and the analysis is similar to the proof of Claim~\ref{cl441}. We leave it to the reader.
\end{proof}

At the end of the above process, for every $\eta\in{}^\kappa2$, we have obtained a streamlined tree $T^\eta:=\bigcup_{\alpha<\kappa}L^{\eta\restriction\alpha}$
whose $\alpha^{\text{th}}$ level is $L^{\eta\restriction\alpha}$.

Using $\diamondsuit^*_S(\kappa\textup{-trees})$,
for each $\eta\in{}^\kappa2$, fix a club $D^\eta\s\kappa$ such that,
for every $\alpha\in S\cap D^\eta$, for every $f\in (T^\eta)_\alpha=L^{\eta\restriction\alpha}$,
the set $\{ \beta<\alpha\mid f\restriction\beta=f_\beta\}$ is stationary in $\alpha$.

\begin{claim}\label{claim442}
Let $\eta\in{}^\kappa2$ and $\alpha\in S\cap D^\eta$.
Then $L^{\bar\eta}=\{\mathbf{b}_x^{\bar\eta}\mid x\in T^{\bar\eta}\restriction C_\alpha\}$, for $\bar\eta:=\eta\restriction\alpha$.
\end{claim}
\begin{proof} The proof is similar to that of Claim~\ref{claim432}, and is left to the reader.
\end{proof}

To see that the family of trees $\langle T^\eta\mid \eta\in{}^\kappa2\rangle$ is as sought,
let $\langle \eta_j\mid j<\chi\rangle$
be an injective sequence of elements of ${}^\kappa2$, with $0<\chi<\kappa$.
Let $\mathbf T=(T,<_T)$ denote he product tree $\bigotimes_{j<\sigma}T^{\eta_j}$.
As $\chi$ is smaller than our strongly inaccessible cardinal $\kappa$, $\mathbf T$ is a (splitting, normal) $\kappa$-tree.
Thus, to show that $\mathbf T$ is a $\kappa$-Souslin tree, it suffices to establish that it has no antichains of size $\kappa$.
To this end, let $A$ be a maximal antichain in $\mathbf  T$.

Set $\Omega:=\{(\langle \eta_j\restriction\epsilon\mid j<\chi\rangle,A\cap{}^\chi(^\epsilon2))\mid \epsilon<\kappa\}$.
As an application of $\diamondsuit(H_\kappa)$,
using the parameter $p:=\{\phi, A,\Omega,\langle T^{\eta_j}\mid j<\chi\rangle\}$,
we get that for every $i<\kappa$, the following set is cofinal (in fact, stationary) in $\kappa$:
$$B_i:=\{\beta\in R_i\cap\acc(\kappa)\mid \exists \mathcal M\prec H_{\kappa^+}\,(p\in \mathcal M, \mathcal M\cap\kappa=\beta, \Omega_\beta=\Omega\cap \mathcal M)\}.$$

Note that, for every $\beta\in \bigcup_{i<\kappa}B_i$,
it is the case that $T\restriction\beta\s\phi[\beta]$.

Fix a large enough $\delta<\kappa$ for which the map $j\mapsto\eta_j\restriction\delta$ is injective over $\chi$.
By the choice of $\vec C$, we may now find a regular cardinal $\alpha\in  S\cap \bigcap_{j<\chi}D^{\eta_j}$ above $\max\{\chi,\delta\}$ such that, for all $i<\alpha$,
$$\sup(\nacc(C_\alpha)\cap B_i)=\alpha.$$

In particular, $T\restriction\alpha\s\phi[\alpha]$.
Set $\bar\eta_j:=\eta_j\restriction\alpha$ for each $j<\chi$,
and note that $T\restriction\alpha=\bigotimes_{j<\chi}T^{\bar\eta_j}$.

\begin{claim}\label{c444}	 $A\s T\restriction\alpha$. In particular, $|A|<\kappa$.
\end{claim}
\begin{proof}
It suffices to show that every element of $T_\alpha$ extends some element of the antichain $A$.
To this end, let $\vec y=\langle y_j\mid j<\chi\rangle$ be an arbitrary element of $T_\alpha$.
For each $j<\chi$, since $\alpha\in S\cap D^{\eta_j}$,
Claim~\ref{claim442} implies that we may find some $x_j\in T^{\bar\eta_j}\restriction C_\alpha$ such that $y_j=\mathbf b_{x_j}^{\bar\eta_j}$.
By Claim~\ref{c441} and the fact that $\cf(\alpha)=\alpha>\chi$,
we may assume the existence of a large enough $\gamma<\alpha$ such that $\dom(x_j)=\gamma$ for all $j<\chi$.
In particular, $\vec x:=\langle x_j\mid j<\chi\rangle$ is an element of $T\restriction\alpha\s\phi[\alpha]$.
Fix some $i<\alpha$ such that $\phi(i)=\vec x$,
and then pick a large enough $\beta\in \nacc(C_\alpha)\cap B_i$ for which $\beta^-:=\sup(C_\alpha\cap\beta)$ is bigger than $\max\{\gamma,\delta\}$.
Note that $\psi(\beta)=\phi(\pi(\beta))=\phi(i)=\vec x$
and that $\langle \bar\eta_j\restriction\beta\mid j<\chi\rangle$ is an injective sequence.

Let $\mathcal M\prec H_{\kappa^+}$ be a witness for $\beta$ being in $B_i$.
Clearly,
\begin{itemize}
\item $T\cap\mathcal M=T\restriction\beta=\bigotimes_{j<\chi}T^{\bar\eta_j\restriction\beta}$,
\item $A\cap\mathcal M=A\cap(T\restriction\beta)$ is a maximal antichain in $T\restriction\beta$, and
\item $\Omega_\beta=\Omega\cap \mathcal M=\{(\langle \eta_j\restriction\epsilon\mid j<\chi\rangle,A\cap{}^\chi(^\epsilon2))\mid \epsilon<\beta\}$.
\end{itemize}

It thus follows that for every $j<\chi$,
$b^{\alpha,\bar\eta_j\restriction\beta}_{x_j}=\vec t(j)$, where $\vec t=\min(Q^{\alpha,\beta},\lhd)$.
In particular, we may fix some $\vec s\in A$ such that, for every $j<\chi$,
$$(\vec s(j)\cup b^{\alpha,\bar\eta_j\restriction\beta^-}_{x_j})\s\vec t(j)=b^{\alpha,\bar\eta_j\restriction\beta}_{x_j}\s\mathbf b_{x_j}^{\bar\eta_j}=y_j.$$
So $\vec s<_T \vec y$. As $\vec s$ is an element of $A$, we are done.
\end{proof}
This completes the proof.
\end{proof}

Putting fullness aside for a moment, we remark that the proof of Theorem~\ref{thm44} should make it clear that
if $\p(\kappa,\kappa,{\sq},\kappa,\{E^\kappa_{\ge\chi}\},2)$ holds for a regular uncountable cardinal $\kappa$
and a cardinal $\chi<\kappa$ such that $\lambda^{<\chi}<\kappa$ for all $\lambda<\kappa$,
then there is a family of $2^\kappa$ many (binary/prolific, normal, streamlined) $\kappa$-Souslin trees such that the product of less than $\chi$ many of them is (either empty or) Souslin.
The details will appear in \cite{paper65}.

\section{Full Souslin trees at successors of regulars}\label{succ}

In this section we provide sufficient conditions for the existence of full $\kappa$-Souslin tree for $\kappa$ a successor of a regular cardinal. Unlike the previous section,
here we open with the most general construction.
\begin{thm}\label{thm51} 	Suppose that:
\begin{itemize}
\item $\kappa=\lambda^+=2^\lambda$ for $\lambda$ a regular uncountable cardinal;
\item $\square^B_\lambda$ and $\diamondsuit(\lambda)$ both hold;
\item  $\p^-(\kappa,2,{\sq_\lambda},\kappa,\{E^\kappa_\lambda\})$ holds.
\end{itemize}

Then there is a family $\mathcal T$ of $2^\kappa$ many streamlined, normal, binary, splitting, full $\kappa$-trees
such that $\bigotimes\mathcal T'$ is $\kappa$-Souslin for every nonempty $\mathcal T'\in[\mathcal T]^{<\lambda}$.
\end{thm}
\begin{proof} Let $\vec D=\langle D_\beta\mid\beta<\kappa\rangle$ be a $\lambda$-bounded $\sq_\lambda$-coherent $C$-sequence over $\kappa$.
Using $\diamondsuit(\lambda)$, fix a sequence of functions $\vec h=\langle h_{\beta}:{\beta}\rightarrow {\beta}\mid {\beta}<\lambda\rangle$
such that, for every $h:\lambda\rightarrow\lambda$, the set $\{{\beta}<\lambda\mid h\restriction {\beta}=h_{\beta}\}$ is stationary.
As $2^\lambda=\lambda^+$, the main result of \cite{Sh:922} implies that $\diamondsuit(\lambda^+)$ holds. Thus,
using Fact~\ref{def_Diamond_H_kappa}, fix sequences $\langle \Omega_\beta\mid\beta<\kappa\rangle$ and $\langle R_i\mid i<\kappa\rangle$
together witnessing $\diamondsuit(H_\kappa)$.
Fix a sequence $\vec C=\langle C_\alpha\mid \alpha<\kappa\rangle$ witnessing $\p^-(\kappa,2,{\sq_\lambda},\kappa,\{E^\kappa_\lambda\})$.
Without loss of generality, $0\in C_\alpha$ for all nonzero $\alpha<\kappa$.
Put
\begin{itemize}
\item $B:=\acc(\kappa)\cap E^\kappa_{<\lambda}$, and
\item $\Gamma:=E^\kappa_\lambda\cup\bigcup\{ \acc(C_\alpha)\mid \alpha\in E^\kappa_\lambda\}$,
\end{itemize}
and note that for all $\alpha\in\Gamma$ and $\beta\in\acc(C_\alpha)$, it is the case that $\beta\in\Gamma\cap B$ and $C_\beta\sq C_\alpha$.

Let $\pi:\kappa\rightarrow\kappa$ be such that $\alpha\in R_{\pi(\alpha)}$ for all $\alpha<\kappa$.
As $2^{<\kappa}=\kappa$, we may let $\lhd$ be some well-ordering of $H_\kappa$ of order-type $\kappa$,
and let $\phi:\kappa\leftrightarrow H_\kappa$ witness the isomorphism $(\kappa,\in)\cong(H_\kappa,\lhd)$.
Put $\psi:=\phi\circ\pi$.

\medskip

We shall follow the construction of Theorem~\ref{thm44} as much as possible,
where the main difference is that instead of assuming $\diamondsuit^*_{S}(\kappa\textup{-trees})$,
for each $\eta\in{}^\kappa2$,
we shall gradually identify a sequence $\vec f^{\eta}:=\langle f^{\eta\restriction\beta}\mid \beta\in B\rangle$
that will play the role of an optimal $\diamondsuit^*_{E^\kappa_\lambda,B}(\kappa\textup{-trees})$-sequence
relativized to the outcome tree $T^\eta$.

Altogether, we shall construct a system $\langle (L^\eta,f^\eta)\mid \eta\in {}^{<\kappa}2\rangle$ such that,
for all $\alpha<\kappa$ and $\eta\in {}^\alpha2$:
\begin{enumerate}
\item[(i)] $L^\eta\in[{}^\alpha2]^{\le\lambda}$,
and we shall fix some enumeration $\langle l^\eta_i\mid i<\lambda\rangle$ of $L^\eta$;

\item[(ii)] for every $\beta<\alpha$, $L^{\eta\restriction\beta}=\{t\restriction\beta\mid t\in L^\eta\}$;
\item[(iii)] $f^\eta:=\bigcup\{ l_i^{\eta\restriction\delta}\mid \delta\in D_\alpha, i=h_{\otp(D_\alpha)}(\otp(D_\alpha\cap\delta))\}$ if $\alpha\in B$,\footnote{To demystify the definition of $f^\eta$, recall the proof of Proposition~\ref{cluba}.}
and $f^\eta:=\emptyset$, otherwise.

\end{enumerate}

By the same convention of the proof of Theorem~\ref{thm44},
for every $\alpha\in\acc(\kappa)$ such that $\langle L^\eta\mid \eta\in {}^{<\alpha}2\rangle$ has already been defined,
for every $\eta\in {}^\alpha2$, we shall denote $T^\eta:=\bigcup_{\beta<\alpha}L^{\eta\restriction\beta}$.

We now turn to the actual construction of the system $\langle (L^\eta,f^\eta)\mid \eta\in {}^{<\kappa}2\rangle$.
The construction is by recursion over $\dom(\eta)$.
We start by letting $L^\emptyset:=\{\emptyset\}$.
For every $\alpha<\kappa$ such that $\langle L^\eta\mid \eta\in {}^\alpha2\rangle$ has already been defined to satisfy our promises,
for every $\eta\in {}^{\alpha+1}2$, let
$$L^\eta:=\{ t{}^\smallfrown\langle0\rangle,t{}^\smallfrown\langle1\rangle\mid t\in L^{\eta\restriction\alpha}\}.$$
For all $\eta\in {}^{\alpha+1}2$, it is clear that $|L^\eta|\le\lambda$, so
we may fix some enumeration $\langle l^\eta_i\mid i<\lambda\rangle$ of $L^\eta$ as dictated by Ingredient~(i).

Suppose now that $\alpha\in\acc(\kappa)$ is such that $\langle L^\eta\mid \eta\in {}^{<\alpha}2\rangle$ has already been defined.
In particular, for every $\eta\in{}^{<\alpha}2\cup{}^\alpha2$, the object $f^\eta$ is determined by Ingredient~(iii).

If $\alpha\in\Gamma$, then for all $\eta\in {}^\alpha2$ and $x\in T^\eta\restriction C_\alpha$,
we shall soon identify some element $\mathbf b_x^{\eta}$ of $\mathcal B(T^\eta)$ extending $x$.
For all $\eta\in {}^\alpha2$, we promise to let:
\begin{equation}\tag{$\star$}\label{promise5}L^\eta:=\begin{cases}\mathcal B(T^\eta),&\text{if }\alpha\notin\Gamma\text{ or if }f^\eta=\mathbf{b}_x^\eta\text{ for some }x\in T^\eta\restriction C_\alpha;\\
\mathcal B(T^\eta)\setminus\{f^\eta\},&\text{otherwise}.\end{cases}\end{equation}

In particular, if $\alpha\notin\Gamma$, then our definition of $\langle L^\eta\mid \eta\in {}^\alpha2\rangle$ is complete (pun intended).
Next, suppose that $\alpha\in\Gamma$. As in the proof of Theorem~\ref{thm44}, we shall first define a matrix
$$\mathbb B^\alpha=\langle b_x^{\alpha,\eta}\mid \beta\in C_\alpha, \eta\in {}^\beta2, x\in T^\eta\restriction C_\alpha\cap(\beta+1) \rangle$$
ensuring that $x\s b_x^{\alpha,\bar\eta}\s b_x^{\alpha,\eta}\in L^\eta$ whenever $\bar\eta\s\eta$.
Then, for all $\eta\in {}^\alpha2$ and $x\in T^\eta\restriction C_\alpha$,
it will follow that $\mathbf b_x^{\eta}:=\bigcup_{\beta\in C_\alpha\setminus\dom(x)}b_x^{\alpha,\eta\restriction\beta}$ is an element of $\mathcal B(T^\eta)$ extending $x$,
and this is the element we will be using when defining $L^\eta$ as per \eqref{promise5}.

We now turn to the recursive construction of the matrix $\mathbb B^\alpha$.
Suppose that $\beta\in C_\alpha$ is such that
$$\mathbb B^\alpha_{<\beta}=\langle b_x^{\alpha,\eta}\mid \bar\beta\in C_\alpha\cap\beta, \eta\in {}^{\bar\beta}2, x\in T^\eta\restriction C_\alpha\cap(\bar\beta+1) \rangle$$
has already been defined.

$\br$ For all $\eta\in {}^\beta2$ and $x\in T^\eta$ such that $\dom(x)=\beta$, let $b_x^{\alpha,\eta} := x$.

$\br$ For all $\eta\in {}^\beta2$ and $x\in T^\eta$ such that $\dom(x)<\beta$,
there are two main cases to consider:

$\br\br$ Suppose that $\beta\in \nacc(C_\alpha)$ and denote $\beta^-:=\sup(C_\alpha\cap\beta)$.

$\br\br\br$ If	$\beta\in \acc(\kappa)$
and there exists a nonzero cardinal $\chi<\lambda$ such that all of the following hold:
\begin{enumerate}
\item There exists a sequence $\langle \eta_j\mid j<\chi\rangle$ of elements of ${}^\beta2$,
and a maximal antichain $A$ in the product tree $\bigotimes_{j<\chi}T^{\eta_j}$ such that
$\Omega_\beta=\{(\langle \eta_j\restriction\epsilon\mid j<\chi\rangle,A\cap{}^\chi(^\epsilon2))\mid \epsilon<\beta\}$;
\item 				$\psi(\beta)$ is a sequence $\langle x_j\mid j<\chi\rangle$ such that $x_j\in T^{\eta_j\restriction\beta^-}\restriction(C_\alpha\cap\beta^-)$ for every $j<\chi$;
\item There exists a unique $j<\chi$ such that $\eta_j=\eta$ and $x_j=x$.
\end{enumerate}

In this case, by Clauses (1) and (2), the following set is nonempty
$$Q^{\alpha,\beta} := \{ \vec t\in \prod\nolimits_{j<\chi}L^{\eta_j}\mid \exists\vec s\in A\forall j<\chi[ (\vec s(j)\cup b^{\alpha,\eta_j\restriction\beta^-}_{x_j})\s\vec t(j)]\},$$
so we let $\vec t:=\min(Q^{\alpha,\beta},\lhd)$,
and then we let $b^{\alpha,\eta}_x:=\vec t(j)$ for the unique index $j$ of Clause~(3).
It follows that $b^{\alpha,\eta\restriction\beta^-}_{x}\s\vec t(j)= b^{\alpha,\eta}_x$.

$\br\br\br$ Otherwise, let $b_x^{\alpha,\eta}$ be the $\lhd$-least element of $L^{\eta}$ extending $b_x^{\alpha,\eta\restriction\beta^-}$.

$\br\br$ Suppose that $\beta\in\acc(C_\alpha)$. Then we define $b_x^{\alpha,\eta} := \bigcup\{b_x^{\alpha,\eta\restriction\bar\beta}\mid \bar\beta\in C_\alpha\cap\beta\setminus\dom(x)\}$.
As $\beta\in\acc(C_\alpha)$ and $\alpha\in\Gamma$, we get that $\beta\in\Gamma$ and $C_\beta=C_\alpha\cap\beta$.
A verification similar to that of Claim~\ref{cl441} yields that $\mathbb B^\alpha_{<\beta}=\mathbb B^\beta$
and $b_x^{\alpha,\eta}=\mathbf{b}_x^{\eta}$ so that the former indeed belongs to $L^{\eta}$.

This completes the definition of the matrix $\mathbb B^\alpha$,
from which we derive $\mathbf b_x^\eta$ for all $\eta\in {}^\alpha2$ and $x\in T^\eta\restriction C_\alpha$,
and then we define $L^\eta$ by adhering to \eqref{promise2}.

\begin{claim}\label{claim511} For all $\eta\in{}^\alpha2$ and $t\in \{\mathbf{b}_x^\eta\mid x\in T^\eta\restriction C_\alpha\}$, there exists a tail of $\varepsilon\in C_\alpha$ such that $t=\mathbf b^\eta_{t\restriction\varepsilon}$.
\end{claim}
\begin{proof} This follows from the canonical nature of the construction,
and is left to the reader.
\end{proof}

It will follow from the upcoming claim that for every $\eta\in{}^\alpha2$, $|L^\eta|\le\lambda$,
so we may fix an enumeration $\langle l^\eta_i\mid i<\lambda\rangle$ of $L^\eta$,
as dictated by Ingredient~(i).

\begin{claim}\label{claim512} Let $\eta\in{}^\alpha2$.
\begin{itemize}
\item[(a)] $|L^\eta|\le\lambda$;
\item[(b)] If $\alpha\in E^\kappa_\lambda$, then for every $\rho\in L^\eta$,
$\{\beta<\alpha \mid \rho\restriction\beta=f^{\eta\restriction\beta}\}$ is stationary in $\alpha$;
\item[(c)] If $\alpha\in E^\kappa_\lambda$, then $L^\eta=\{\mathbf{b}_x^\eta\mid x\in T^\eta\restriction C_\alpha\}$.
\end{itemize}
\end{claim}
\begin{proof} (a) By Ingredient~(i) thus far, $|\bigcup_{\beta<\alpha}L^{\eta\restriction\beta}|\le\lambda$. Therefore:
\begin{itemize}
\item[$\br$] If $\cf(\alpha)<\lambda$, then $\mathcal B(T^\eta)\le\lambda^{\cf(\alpha)}\le\lambda^{<\lambda}=\lambda$,
since $\diamondsuit(\lambda)$ holds.
\item[$\br$] If $\cf(\alpha)=\lambda$, then the conclusion will follow from clause~(c) below.
\end{itemize}

(b) Suppose $\alpha\in E^\kappa_\lambda$, and let $\pi_\alpha:\lambda\rightarrow D_\alpha$ denote the inverse collapsing map of $D_\alpha$.
Now, given $\rho\in L^\eta$, by Ingredient~(ii), for every $\Lambda<\lambda$, $\rho\restriction\pi_\alpha(\Lambda)$ is in $L^{\eta\restriction\pi_\alpha(\Lambda)}$,
so we may define a function $h:\lambda\rightarrow\lambda$ via:
$$h(\Lambda):=\min\{ i<\lambda\mid \rho\restriction \pi_\alpha(\Lambda)=l_i^{\eta\restriction\pi_\alpha(\Lambda)}\}.$$
By the choice of $\vec h$, the set $b:=\{\bar\beta\in\acc(\lambda)\mid h\restriction\bar\beta=h_{\bar\beta}\}$ is stationary.
Consequently, the set $B^*:=\pi_\alpha[b]$ is a stationary subset of $B\cap\alpha$.
Let $\beta\in B^*$. Pick ${\bar\beta}\in\acc(\lambda)$ such that $\beta=\pi_\alpha({\bar\beta})$.
Then $\beta\in\acc(D_\alpha)\s B$. As $\vec D$ is $\sq_\lambda$-coherent,
$D_\beta=D_\alpha\cap\beta$,
and hence Ingredient~(iii) yields that:
\begin{align*}
f^{\eta\restriction\beta}&=\bigcup\{ l^{\eta\restriction\delta}_i\mid \delta\in D_\beta, i=h_{\otp(D_\beta)}(\otp(D_\beta\cap\delta))\}\\
&=\bigcup\{ l^{\eta\restriction\delta}_i\mid \delta\in D_\beta, i=h_{{\bar\beta}}(\otp(D_\beta\cap\delta))\}\\
&=\bigcup\{ l^{\eta\restriction\delta}_i\mid \delta\in D_\beta, i=h(\otp(D_\beta\cap\delta))\}\\
&=\bigcup\{ l^{\eta\restriction\pi_\alpha(\Lambda)}_i\mid \Lambda<{\bar\beta}, i=h(\Lambda)\}\\
&=\bigcup\{ \rho\restriction\pi_\alpha(\Lambda)\mid \Lambda<{\bar\beta}\}\\
&=\bigcup\{ \rho\restriction\delta\mid \delta\in D_\beta\}\\
&=\rho\restriction\beta,
\end{align*}
as sought.

(c) Suppose $\alpha\in E^\kappa_\lambda$. Now, given $\rho\in L^\eta$,
by Clause~(b), the following set is stationary in $\alpha$:
$$B_\rho:=\{\beta\in\acc(C_\alpha) \mid \rho\restriction\beta=f^{\eta\restriction\beta}\}.$$

Note that $B_\rho\s\acc(C_\alpha)\s\Gamma$.
So, by \eqref{promise5},
for every $\beta\in B_\rho$,
since $f^{\eta\restriction\beta}=\rho\restriction\beta$ is in $L^{\eta\restriction\beta}$,
there must exist some $x\in T^{\eta\restriction\beta}\restriction C_\beta$ such that $f^{\eta\restriction\beta}=\mathbf{b}_{x}^{\eta\restriction\beta}$.
Then, by Claim~\ref{claim511} and Fodor's lemma for ordinals of uncountable cofinality, we may fix a large enough $\varepsilon\in C_\alpha$ such that
$B_\rho^\varepsilon:=\{ \beta\in B_\rho\mid f^{\eta\restriction\beta}=\mathbf{b}_{\rho\restriction\varepsilon}^{\eta\restriction\beta}\}$ is stationary.
Denote $x:=\rho\restriction\varepsilon$, so that $\rho\restriction\beta=\mathbf b^{\eta\restriction\beta}_{x}$ for every $\beta\in B_\rho^\varepsilon$.
Furthermore, for every $\beta\in B_\rho^\varepsilon$,
since $C_\alpha\cap\beta=C_\beta$,
a verification similar to that of Claim~\ref{cl441} (i.e., $\mathbb B^\alpha_{<\beta}=\mathbb B^\beta$)
implies that $\mathbf b^{\eta}_{x}\restriction\beta=\mathbf b^{\eta\restriction\beta}_{x}$.
Altogether, $\rho=\mathbf b_x^\alpha$.
\end{proof}

At the end of the above process, for every $\eta\in{}^\kappa2$, we have obtained a streamlined tree $T^\eta:=\bigcup_{\alpha<\kappa}L^{\eta\restriction\alpha}$
whose $\alpha^{\text{th}}$ level is $L^{\eta\restriction\alpha}$.

To see that the family of trees $\langle T^\eta\mid \eta\in {}^\kappa2\rangle$ is as sought,
let $\chi<\lambda$ be a nonzero cardinal,
and fix a sequence $\langle \eta_j\mid j<\chi\rangle$ of elements of ${}^\kappa2$.
Let $\mathbf T=(T,<_T)$ denote the product tree $\bigotimes_{j<\sigma}T^{\eta_j}$.
As $\lambda^{\chi}=\lambda<\kappa$, $\mathbf T$ is a (splitting, normal) $\kappa$-tree.
Thus, to show that $\mathbf T$ is a $\kappa$-Souslin tree, it suffices to establish that it has no antichains of size $\kappa$.
To this end, let $A$ be a maximal antichain in $\mathbf  T$.

Set $\Omega:=\{(\langle \eta_j\restriction\epsilon\mid j<\chi\rangle,A\cap{}^\chi(^\epsilon2))\mid \epsilon<\kappa\}$.
As an application of $\diamondsuit(H_\kappa)$,
using the parameter $p:=\{\phi,A,\Omega,\langle T^{\eta_j}\mid j<\chi\rangle\}$,
we get that for every $i<\kappa$, the following set is cofinal (in fact, stationary) in $\kappa$:
$$B_i:=\{\beta\in R_i\cap\acc(\kappa)\mid \exists \mathcal M\prec H_{\kappa^+}\,(p\in \mathcal M, \mathcal M\cap\kappa=\beta, \Omega_\beta=\Omega\cap \mathcal M)\}.$$

Fix a large enough $\delta<\kappa$ for which the map $j\mapsto\eta_j\restriction\delta$ is injective over $\chi$.
By the choice of $\vec C$, we may find an ordinal $\alpha\in  E^\kappa_\lambda$ above $\delta$ such that, for all $i<\alpha$,
$$\sup(\nacc(C_\alpha)\cap B_i)=\alpha.$$

Set $\bar\eta_j:=\eta_j\restriction\alpha$ for each $j<\chi$,
and note that $T\restriction\alpha=\bigotimes_{j<\chi}T^{\bar\eta_j}$.
\begin{claim}	 $A\s T\restriction\alpha$. In particular, $|A|<\kappa$.
\end{claim}
\begin{proof} This is very similar to the proof of Claim~\ref{c444}. Let $\vec y=\langle y_j\mid j<\chi\rangle$ be an arbitrary element of $T_\alpha$,
and we shall show that it extends some element of $A$.
For each $j<\chi$, since $\alpha\in E^\kappa_\lambda$,
Claim~\ref{claim512}(c) implies that we may find some $x_j\in T^{\bar\eta_j}\restriction C_\alpha$ such that $y_j=\mathbf b_{x_j}^{\bar\eta_j}$.
By Claim~\ref{claim511} and $\cf(\alpha)=\lambda>\chi$,
we may assume the existence of $\gamma<\alpha$ such that $\dom(x_j)=\gamma$ for all $j<\chi$.
In particular, we may find some $i<\alpha$ such that $\phi(i)$
is equal to $\vec x:=\langle x_j\mid j<\chi\rangle$.
Pick $\beta\in \nacc(C_\alpha)\cap B_i$ for which $\beta^-:=\sup(C_\alpha\cap\beta)$ is bigger than $\max\{\gamma,\delta\}$.
Then:
\begin{itemize}
\item $\psi(\beta)=\vec x$,
\item $\langle \bar\eta_j\restriction\beta\mid j<\chi\rangle$ is injective,
\item $\Omega_\beta=\{(\langle \bar\eta_j\restriction\epsilon\mid j<\chi\rangle,A\cap{}^\chi(^\epsilon2))\mid \epsilon<\beta\}$, and
\item $A\cap{}^{<\beta}2$ is a maximal antichain in $T\restriction\beta=\bigotimes_{j<\chi}T^{\bar\eta_j\restriction\beta}$.
\end{itemize}

It thus follows that for every $j<\chi$,
$b^{\alpha,\bar\eta_j\restriction\beta}_{x_j}=\vec t(j)$, where $\vec t=\min(Q^{\alpha,\beta},\lhd)$.
In particular, we may fix some $\vec s\in A$ such that, for every $j<\chi$,
$$(\vec s(j)\cup b^{\alpha,\bar\eta_j\restriction\beta^-}_{x_j})\s\vec t(j)=b^{\alpha,\bar\eta_j\restriction\beta}_{x_j}\s\mathbf b_{x_j}^{\bar\eta_j}=y_j.$$
So $\vec y$ extends an element of $A$, as sought.
\end{proof}
This completes the proof.
\end{proof}
\begin{remark} \begin{enumerate}
\item Claim~\ref{claim512}(c) implies that the trees constructed above fall into the class of trees obtained using the \emph{microscopic approach} as a transfinite application of \emph{actions}
to control various features of the outcome trees (see \cite[Definition~6.5]{paper23}).
By embedding additional calls for actions, we can easily ensure that the above full $\lambda^+$-Souslin trees be $\lambda$-free or specializable via a $\lambda$-closed $\lambda^{+}$-cc notion forcing.
\item The hypotheses of Theorem~\ref{thm51} are all compatible with $\lambda$ being supercompact.
Indeed, starting with a Laver-indestructible supercompact $\lambda$,
first use Baumgartner's $\lambda$-directed-closed notion of forcing to add a $\square^B_\lambda$-sequence,
and then force with $\add(\lambda^+,1)$ to arrange $2^\lambda=\lambda^+$.
Finally, use \cite[Definition~3.16]{paper28} to add a $\p^-(\kappa,2,{\sq_\lambda},\kappa,\{E^\kappa_\lambda\})$-sequence via a $\lambda$-directed closed $\kappa$-strategically closed forcing.
Since every supercompact cardinal is subtle, $\diamondsuit(\lambda)$ will hold for free.
\end{enumerate}
\end{remark}

We now arrive at Theorem~\ref{thmc}:

\begin{cor} 	Suppose that:
\begin{itemize}
\item $\kappa=\lambda^+$ for $\lambda$ a regular uncountable cardinal;
\item $\sd_\lambda$ and $\diamondsuit(\lambda)$ both hold.
\end{itemize}

Then there exists a sequence  $\langle T^\eta \mid\eta<2^\kappa\rangle$ of
streamlined, normal, binary, splitting, full $\kappa$-trees
such that for every nonzero cardinal $\chi<\lambda$, for every	injective sequence $\langle \eta_j\mid j<\chi\rangle$ of elements of $2^\kappa$, the product tree $\bigotimes_{j<\chi}T^{\eta_j}$ is $\kappa$-Souslin.
\end{cor}
\begin{proof} $\sd_\lambda$ implies $\ch_\lambda$.
It also implies $\square_\lambda$ that implies $\square_\lambda^B$.
By \cite[Theorem~3.6]{paper22}, $\sd_\lambda$ implies $\p(\kappa,2,{\sq},\kappa,\{E^\kappa_\lambda\})$.
Now, appeal to Theorem~\ref{thm51}.
\end{proof}

As said before, we opened this section with the most general construction.
If all one wants is a single full $\kappa$-Souslin tree (e.g., with $\kappa=\aleph_2$), then this may be obtained from the following relaxed hypotheses.\footnote{By Proposition~\ref{friedman}, the hypothesis is moreover optimal.}

\begin{thm}\label{thm53}
Suppose that:
\begin{enumerate}
\item $\kappa=\lambda^+=2^\lambda$ for a regular uncountable cardinal $\lambda$;
\item $\square_\lambda$ and $\diamondsuit(\lambda)$ both hold.
\end{enumerate}

Then there exists a streamlined, normal, prolific, full $\kappa$-Souslin tree.
\end{thm}
\begin{proof} A variation of the proof of Theorem~\ref{thm44} in the spirit of that of Theorem~\ref{thm51} yields that
a streamlined, normal, prolific full $\kappa$-Souslin tree may be constructed under the hypothesis (1) together with the following two:
\begin{itemize}
\item[($2'$)] $\square^B_\lambda$ and $\diamondsuit(\lambda)$ both hold;
\item[($3'$)] Either $\p^-(\kappa,2,{\sq_\lambda},1,\{E^\kappa_\lambda\})$ or $\p^-(\kappa,2,{\sq^*},1,\{E^\kappa_\lambda\})$ holds.
\end{itemize}
It is clear that hypothesis (2) implies ($2'$),
thus we just need to show that ($3'$) follows from $(1)$ and $(2)$.
This is indeed the case, since, by \cite[Corollary~3.1]{paper22} combined with \cite[Lemma~3.8]{paper32},
for every regular uncountable cardinal $\lambda$,
$\square_\lambda\mathrel{+}\ch_\lambda$ implies that $\p(\lambda^+,2,{\sq^*},1,\{E^{\lambda^+}_\lambda\})$ holds.
\end{proof}
By Observation~\ref{obs22}, if a full $\kappa$-Aronszajn tree exists for a successor cardinal $\kappa=\lambda^+$,
then $\lambda=\lambda^{<\lambda}$ (in particular, $\lambda$ is a regular cardinal), and $\lambda\ge\mathfrak c$.
We now arrive at Theorem~\ref{thmb}.

\begin{cor} Suppose that $\lambda=\lambda^{<\lambda}$ is a successor of an uncountable cardinal such that
$\square_\lambda\mathrel{+}\ch_\lambda$ holds. Then there exists a full $\lambda^+$-Souslin tree.
\end{cor}
\begin{proof} By the main result of \cite{Sh:922}, $\diamondsuit(\lambda)$ holds.
Now, appeal to Theorem~\ref{thm53}.
\end{proof}

\section*{Acknowledgments}

The first and second author were supported by the Israel Science Foundation (grant agreement 203/22).
The first and third author were supported by the European Research Council (grant agreement ERC-2018-StG 802756).

Some of the results of this paper were announced by the first author
at the \emph{Set Theory} online workshop in Oberwolfach, January 2022,
and by the second author in a poster session at the \emph{Advances in Set Theory 2022} conference in Jerusalem, July 2022.
We thank the corresponding organizers for the opportunity to present this work and the participants for their feedback.

\end{document}